\documentclass[10pt,reqno]{amsart}
\usepackage{amsmath,amssymb,amsfonts,amsthm,epsfig,color,graphicx,eucal,mathrsfs}
\usepackage{graphicx,enumerate,bbold}
\usepackage{physics}
\usepackage{epstopdf}
\usepackage{float}
\usepackage[labelfont=small]{subcaption}
\usepackage[ruled,vlined]{algorithm2e}
\usepackage{marvosym}
\usepackage{multirow,url}
\usepackage{fancyhdr}
\usepackage[german,ngerman,english]{babel}

\usepackage{footnote}
\usepackage{footmisc}
\newcommand{\astfootnote}[1]{
  \let\oldthefootnote=\thefootnote
  \setcounter{footnote}{0}
  \renewcommand{\thefootnote}{\fnsymbol{footnote}}
  \footnotetext{#1}
  \let\thefootnote=\oldthefootnote
}

\usepackage[
          breaklinks=true,
          colorlinks=true,
          linkcolor=red,
          citecolor=green,
          urlcolor=blue
            ]{hyperref}

\numberwithin{equation}{section}
\numberwithin{figure}{section}
\theoremstyle{plain}
\newtheorem{theorem}{Theorem}[section]
\newtheorem{lemma}[theorem]{Lemma}
\newtheorem{proposition}[theorem]{Proposition}
\newtheorem{corollary}[theorem]{Corollary}
\newtheorem{conjecture}{Conjecture}[section]

\theoremstyle{definition}
\newtheorem{definition}{Definition}[section]

\newtheorem{remark}{Remark}[section]

\newcommand{\bitem}{\begin{itemize}}
\newcommand{\eitem}{\end{itemize}}
\newcommand{\mc}[1]{\mathcal{#1}}
\newcommand{\mb}[1]{\mathbb{#1}}

\newcommand{\N}{\mathbb{N}}
\newcommand{\R}{\mathbb{R}}

\newcommand{\B}{\mathbb{B}}

\newcommand{\EE}{\mathbb{E}}

\newcommand{\bpm}{\begin{pmatrix}}
\newcommand{\epm}{\end{pmatrix}}
\newcommand{\bsm}{\left(\begin{smallmatrix}}
\newcommand{\esm}{\end{smallmatrix}\right)}
\newcommand{\T}{\top}

\newcommand{\ol}[1]{\overline{#1}}
\newcommand{\la}{\langle}
\newcommand{\ra}{\rangle}

\newcommand{\mrm}[1]{\mathrm{#1}}

\newcommand{\row}[2]{{#1}_{#2,\bullet}}
\newcommand{\col}[2]{{#1}_{\bullet,#2}}
\newcommand{\veps}{\varepsilon}

\newcommand{\eins}{\mathbb{1}}

\DeclareMathOperator{\dom}{dom}
\DeclareMathOperator{\intr}{int}
\DeclareMathOperator{\rint}{rint}

\DeclareMathOperator{\rbd}{rbd}

\DeclareMathOperator{\argmin}{arg min}

\DeclareMathOperator{\supp}{supp}
\DeclareMathOperator{\cosupp}{cosupp}
\DeclareMathOperator{\aff}{aff}

\DeclareMathOperator{\cone}{cone}

\DeclareMathOperator{\cl}{cl}

\DeclareMathOperator{\sign}{sign}
\DeclareMathOperator{\dist}{dist}

\DeclareMathOperator{\TV}{TV}

\title{Performance Bounds For Co-/Sparse Box Constrained Signal Recovery$^\dagger$}
\author{}

\begin{document}
\maketitle
\thispagestyle{plain}
\begin{center}
{\bf Jan Kuske and Stefania Petra}
\end{center}
\vspace{3mm}
\begin{abstract}
The recovery of structured signals from a few linear measurements is a central point in both compressed sensing (CS) and discrete tomography. In CS the signal structure is described by means of a low complexity model e.g. co-/sparsity.
The CS theory shows that any signal/image can be undersampled at a rate dependent on its intrinsic complexity. Moreover, in such undersampling regimes, the signal can be recovered by sparsity promoting convex regularization like $\ell_1$- or total variation ($\TV$-) minimization.  Precise relations between many low complexity measures and the sufficient number of random measurements are known for many sparsity promoting norms. However, a precise estimate of the undersampling rate for the TV seminorm is still lacking. We address this issue by: a) providing dual certificates testing uniqueness of a given cosparse signal with bounded signal values, b) approximating the undersampling rates via the statistical dimension of the TV descent cone and c) showing empirically that the provided rates also hold for tomographic measurements.
\end{abstract}
\astfootnote{%
  $^\dagger$ This is an invited paper to the special issue of Analele Stiintifice ale Universitatii Ovidius Constanta (ASUOC) on the occasion of the 11th Workshop on Mathematical Modelling of Environmental and Life Sciences Problems organised in October 12-16, 2016, Constanta, Romania (\url{http://www.ima.ro/mmelsp.htm}). It was accepted for publication in 2017 but not yet published.\\
  The authors would like to stress that this is not a recent  paper and is not referencing other works that have appeared since 2017. 
}
\newpage
\tableofcontents
\newpage
%
\section{Introduction}\label{sec:intro}

The current work is motivated by the image reconstruction problem from few linear measurements, e.g. discrete tomography \cite{DTbook}, where typically bounds on the image values are available.
This explains why we are concerned with the recovery of \emph{box constrained} 
signals. Moreover, we aim to reconstruct an image that exhibits some structure, e.g. sparsity/gradient-sparsity.

Recovering a structured signal/image from few linear measurements is a central point in both
compressed sensing (CS) \cite{Rauhut-book} and discrete tomography \cite{DTbook}. In CS the signal structure is described by means of a low complexity model. For instance, signal sparsity $s:=\|x\|_0$, that equals the number of non-zero components in $x$,
and the associated $\ell_0$-minimization is typically relaxed to $\ell_1$-minimization. 
The theory of CS implies that if the number of measurements $m\ge C\cdot s\log(n/s)$ and the entries of the measurement matrix $A$ follow a normal distribution, 
i.e.~$a_{ij} \sim \mc{N}(0,1)$,
the solution of $\ell_0$-minimization  is \emph{equal} to the solution of  $\ell_1$-minimization.
Thus, a combinatorial problem can be solved by a convex optimization problem.

In the present work we consider the noiseless setting
\begin{equation}\label{eq:Ax=b}
A\ol{x}=b,
\end{equation}
where the unknown (structured) signal $\ol{x}\in\R^n$ is undersampled either by
a Gaussian matrix $A\in\R^{m\times n}$, i.e.~$A$ has  independent standard normal entries, or by a  
tomographic projection matrix $A$. 

For the reconstruction of the sparse or gradient-sparse box constrained 
signal of interest we consider some structure enforcing regularizer $f:\R^n\to\ol{\R}$, that is a proper convex function, 
and solve
\begin{equation}\label{eq:minf-exact-constr}
\min_{x} f(x)\qquad {\rm subject~to} \quad Ax=A\ol{x},
 \end{equation}
where $f$ is specialized to one of the following functions 

\begin{subequations} \label{def-f}
\begin{align}
f_1(x) &= \|x\|_1 , \label{def-f1-l1}\\
f_2(x) &= \|x\|_1  + \delta_{\R^n_+}(x) ,  \label{def-f2-l1-nonneg}\\
f_3(x) &= \|x\|_1  + \delta_{{[0,1]}^n}(x)  , \label{def-f3-l1-box}\\
f_4(x) &= \|D x\|_1 , \label{def-f4-TV} \\
f_5(x) &= \|Dx\|_1  + \delta_{\R^n_+}(x) , \label{def-f5-TV-nonneg}\\
f_6(x) &= \|Dx\|_1  + \delta_{{[0,1]}^n}(x).  \label{def-f6-TV-box}
\end{align}
\end{subequations}
In the first part of the paper $D\in\R^{p\times n}$ is an arbitrary matrix, while in the latter part $D$ is  specialized to the discrete gradient operator, 
in view of the (anisotropic) 
$\TV$-regularizer we are interested in.
\subsection{Contribution, Related Work, Organization}
Our main objectives are: 
\begin{description}
\item[{\it Individual Uniqueness}] We derive practical tests for individual uniqueness: 
Given a signal $\ol{x}$ and an instance of the 
optimization problem above, can we efficiently verify whether the signal is the unique solution?
Our approach will be to apply the uniqueness condition $0\in\intr{\partial f (\ol{x})}$ \cite{Gilbert2017} to several types of polyhedral functions $f$
including $\ell_1$- and $\TV$-minimization with and without box constraints, as detailed in
Sect.~\ref{sec:individual-unique}.  Such testable uniqueness conditions for $\TV$-minimization were also
developed in \cite{jorgensen2015testable, Zhang2016}, but these use different mathematical tools. Our derivation is shorter
and  additionally considers  box constrained $\TV$-minimization.
\item[{\it Probabilistic Uniqueness}] We wish to accurately describe the relation between the signal sparsity and the number of measurements that guarantee uniqueness with high probability. We build on recent CS theory 
\cite{Amelunxen2014living, Oymak2016denoising,DonohoJM13,Foygel2014}
that upper bounds the statistical dimension of the descent cone of the structure enforcing regularization.
The difficulty lies in the fact that certain CS guarantees are missing for $\TV$-minimization.
Moreover, the CS recovery performance bounds based on the statistical dimension
(also called phase transitions) do not apply directly to tomography, since
tomographic measurements are known to be ill-conditioned on the
common structured signal classes from CS. 
Phase transitions in CS are discussed in Section \ref{sec:prob-unique}.
\item[{\it Empirical validation}]
We verify the theoretical probabilistic results for Gaussian matrices 
along with the theoretical uniqueness tools and compare them to the recovery and uniqueness performance of typical tomographic measurements. 
We stress that  tomographic projections fall short of the common assumptions (e.g.~restricted isometry property) underlying compressed sensing, see
\cite{Petra2014}, and CS recovery guarantees \emph{do not apply} for tomography. In view of its practical importance (reducing radiation exposure) we emphasize the need of accurately describing how much undersampling is tolerable in tomographic recovery.
We empirically validate our theoretical results in Sect. \ref{sec:experiments}. 
Hence, our work is closely related to the work in \cite{jorgensen2015testable}, but additionally provides
approximations to the \emph{theoretical} phase transition curves that are validated empirically.
\end{description}

\subsection{Notation}
For $n \in \N$, we use the shorthands $[n] = \{1,2,\dotsc,n\}$.
We define $\R_{+}^n:=\{x\in\R^n \colon x_i \ge 0, \forall i\in{[n]}\}$ and
$\R_{++}^n:=\{x\in\R^n \colon x_i > 0, \forall i\in{[n]}\}$. Analogously, we define $\R_{-}^n$ and $\R_{--}^n$.
The extended real line is denoted by $\ol{\R}:=\R\cup \{\pm \infty\}$.
 Vectors $x\in\R^n$ are column vectors and indexed by superscripts. $x^{\T}$ denotes the transposed vector $x$ and
$\la x, y \ra$ or $x^\T y$ the Euclidean inner product. $\T$ stands for the transpose. 
For a vector $x\in  {\R}^{n}$, $\sign (x) \in  {\R}^{n}$ is the \emph{sign vector} of $x$ defined
component-wise   as ${\sign}(x)_{i}=1$ if $x_{i}>0$ and $ {\sign}(x)_{i}=-1$ if $x_{i}<0$.
For some matrix $A\in\R^{m\times n}$,  the nullspace is denoted by $\mc{N}(A)$ and its range by $\mc{R}(A)$.
For some matrix $\Omega\in\R^{p\times n}$ we denote the submatrix indexed by the rows in $\Lambda\subset{[p]}$ by
 $\row{\Omega}{\Lambda}$. Similarly, for some matrix $A\in\R^{m\times n}$ the columns indexed by $S\subset{[n]}$ are denoted by  $\col{A}{S}$.
 The complement of an index set $S\subset{[n]}$ will be denoted by $S^c:={[n]}\setminus S$.
For a subspace $X$, $X^\perp$ will denote its orthogonal complement. The affine hull of a set $X$, see Appendix, is denoted by $\aff (X)$.
The conic hull is denoted by $\cone X$.
We use the $\ell_1$-norm  $\|x\|_1=\sum_{i\in{[n]}}|x_i|$ and the maximum norm $\|x\|_\infty:=\max_{i\in{[n]}}|x_i|$.
The $\ell_0$-measure (not a norm!) stands for the cardinality of the support of $x$, i.e. $\|x\|_0:=|\supp(x)|$, with $\supp(x) =\{i\in{[n]}\colon x_i\ne 0\}$.
The cosupport of $x$ is denoted by $\cosupp(x) =\{i\in{[n]}\colon x_i = 0\}$.

We define the one-dimensional discrete derivative operator ${\partial_n} \colon \R^n \to \R^{n-1}$,  $(\partial_n)_{ij} =-1$ if $i=j$, $(\partial_n)_{ij} =1$ if $j=i+1$
and $(\partial_n)_{ij} =0$ otherwise. For $m\times n$ images the discrete gradient operator is defined as
\begin{equation}\label{eq:def-nabla}
\nabla = \bpm
I_{n} \otimes \partial_{m}  \\
\partial_{n} \otimes I_{m}  \\
\epm
\in \R^{p \times n},
\end{equation}
where $\otimes$ denotes the Kronecker product and $I_{m}, I_{n}$, are identity matrices. The \emph{anisotropic} discretized total variation (TV) is given by $ \TV(x) := \|\nabla x \|_{1} $. The image gradient sparsity is given by $ \|\nabla x \|_{0}$.

We further denote the indicator function of a convex set
as  $\delta_{C}:\R^n\to\ol{\R}$. Recall,  $\delta_{C}(x)=0$ if $x\in C$ and $\delta_{C}(x)=\infty$, 
otherwise. The subdifferential of a function $f$ at $x$ will be denoted by $\partial f(x)$, see Appendix for the definition.
We will use that the subdifferential of the indicator function equals the normal cone $\partial \delta_{C}(x)=N_C(x):=\{v\in\R^n\colon v^\top(y-x)\le 0, \forall y\in C\}$. 
The distance of a vector $x$
to $C$ is denoted by $\dist(x,C)$ and defined as $\dist(x,C)= \min_{y\in C} \|x-y\|$. Finally, $\EE(X)$ will denote the expectation of the random variable $X$.

\section{Individual  Uniqueness and Dual Certificates}\label{sec:individual-unique}
In this section we derive testable uniqueness conditions for our six problems of interest.
We will first provide some recovery conditions called \emph{dual certificates}.
More precisely, consider that we are given a specific vector $\ol{x}$ and we have to decide whether it is the \emph{unique}  solution of
\eqref{eq:minf-exact-constr}.  We will provide \emph{necessary and sufficient} conditions which certify the existence and uniqueness of a solution \eqref{eq:minf-exact-constr} for the case of a polyhedral function $f$. Our analysis closely follows \cite{Gilbert2017}.
These conditions are formulated in terms of a solution to the dual problem of  \eqref{eq:minf-exact-constr}.
For the cases enumerated in \eqref{def-f}, we will see that it is possible to test uniqueness by simply solving a linear program.

In this section, it will be useful to recast \eqref{eq:minf-exact-constr} as an unconstrained optimization problem 
\begin{equation}\label{eq:f-obj-Ax=b}
\min f(x)\quad \text{subject to}\quad Ax=b \quad  \Longleftrightarrow \quad \min_{x\in\R^n} g(x), \quad g(x):= f(x) +\delta_{\mc{X}}(x),
\end{equation}
where $\delta_{\mc{X}}$ denotes the indicator function of the \emph{feasible set} of \eqref{eq:minf-exact-constr} with
\begin{equation}\label{eq:def-feas-set-X}
\mc{X}:=\{x\in  {\R}^{n}:Ax=b\},
\end{equation}
$f$ the objective in \eqref{eq:minf-exact-constr}  and
\begin{equation}\label{eq:def-b}
b:=A \ol{x}.
\end{equation}

Hence, the function $g$ considered in this section is \emph{convex polyhedral}, meaning that its epigraph is a  convex polyhedron.

\subsection{Uniqueness of a Minimizer of a Polyhedral Function}

We will use that the condition $0\in  {\intr}~\partial g(\ol{x})$ is equivalent to the uniqueness of the solution $\ol{x}$, in the case of a polyhedral convex function $g$. The following result gives us a necessary and sufficient condition for the uniqueness of a polyhedral function minimizer
and was recently emphasized in \cite{Gilbert2017}.
\begin{lemma}{\cite[Lem 2.2]{Gilbert2017}} \label{lem:unique-minimizer-polyhedral} 
Let $g\in\mc{F}_0(\R^n)$ be any proper, convex and lsc function that is in addition polyhedral. Then
\begin{equation}\label{eq:lemma-unique-polyhedral}
 \argmin g=\{\ol{x}\}\ \Leftrightarrow\ 0\in  {\intr} ~\partial g(\ol{x}).
\end{equation}
\end{lemma}
We note: the necessity part uses polyhedrality, the sufficiency part is straightforward 
and does not require polyhedrality. We will see, in Sect. \ref{sec:prob-unique}, Prop. \ref{prop:unique_cond_intersection},
an alternative uniqueness condition for the problem \eqref{eq:minf-exact-constr}. The advantage of \eqref{eq:lemma-unique-polyhedral}
is that it directly allows the derivation of testable uniqueness conditions as detailed next.

\subsection{Certifying Recovery of Individual Vectors via Dual Certificates}
In this section we show how to decide if a given solution for \eqref{def-f} is unique. 
To this end, we derive uniqueness conditions for the problem 

\begin{align} \label{eq:f-general-box}
 \min_{x\in\R^n} \|Dx\|_1 + \delta_{\mc{X}}(x) + \delta_{{[l,u]}}(x),
\end{align}
where $l_i\in\R\cup \{-\infty\}$, $u_i\in\R\cup \{+\infty\}$, $l_i < u_i$, $i\in{[n]}$. Hence, \eqref{eq:f-general-box} can be regarded as a generalization of \eqref{def-f1-l1}-\eqref{def-f6-TV-box}. For example \eqref{def-f2-l1-nonneg} can be obtained from \eqref{eq:f-general-box} by setting $D:=I_n$, $l:=0\in\R^n$ and $u_i:=+\infty$, $i\in{[n]}$.
In order to apply the uniqueness criterion of Lem.~\eqref{lem:unique-minimizer-polyhedral} we will need to calculate
the subdifferential of the objective function in \eqref{eq:f-general-box}. To compute the subdifferentials at a feasible point $\ol{x}$ one can apply the sum rule \cite[Cor. 10.9]{rockafellar1wets}

\begin{align}
  \partial \left(  \|D\ol{x}\|_1 + \delta_{\mc{X}}(\ol{x}) + \delta_{{[l,u]}^n}(\ol{x}) \right) =   \partial \|D\ol{x}\|_1 + \partial \delta_{\mc{X}}(\ol{x}) + \partial \delta_{{[l,u]}^n}(\ol{x}).
\end{align}
Since all individual  components are polyhedral functions and $\ol{x}\in\dom \|D\cdot \|_1\cap {\mc{X}} \cap {{[l,u]}^n}$
holds by construction, we can simply add the subdifferentials  of the individual terms
\begin{subequations}
\begin{equation}\label{eq:subdiff-Dl1}
 \partial \|D\ol{x}\|_1=\{ D^\top \alpha \colon \alpha_{\Lambda^c} = \sign ( \row{D}{\Lambda^c} \ol{x} ) , \| \alpha_{\Lambda} \|_\infty \le 1  \} ,
\end{equation}
\begin{equation}\label{eq:subdiff-ind-X}
\partial\delta_{\mc{X}}(\ol{x}) = N_\mc{X}(\ol{x}) =\mc{N}(A)^\perp,
\end{equation}
\begin{equation}\label{eq:subdiff-ind-box}
\partial\delta_{{[l,u]}}(x) = N_{[l,u]}(\ol{x}).
\end{equation}
\end{subequations}

The normal cone corresponding to the box constraints reads
\begin{equation}\label{eq:def-normal-cone-box}
N_{[l,u]}(\ol{x}) = N_{[l_1,u_1]}(\ol{x}_1) \times \dots \times N_{[l_n,u_n]}(\ol{x}_n) =  \R^{|S_l|}_- \times \{0\}^{|S_{lu}|}\times \R^{|S_u|}_{+},
\end{equation}
since 

\begin{equation}
 N_{[l_i,u_i]}(\ol{x}_i) =\begin{cases} 0, & \ol{x}_i \in (l_i,u_i),\\
 (-\infty,0], & \ol{x}_i = l_i,\\
 [0,\infty), & \ol{x}_i = u_i,
 \end{cases}
\end{equation}

where  $S_l=\{i\in {[n]}\colon x_i =l_i \}$, $S_u=\{i\in {[n]}\colon x_i =u_i \}$
and $S_{lu}=(S_l\cup S_u)^c=\{i\in {[n]} \colon x_i\in (l_i,u_i)\}$.
If for example we specialize \eqref{eq:f-general-box} to \eqref{def-f2-l1-nonneg}, the normal cone of the nonnegative
orthant becomes relevant and specializes to 
\begin{equation}\label{normal-come-orthant}
\partial \delta_{\R^n_+}(\ol{x}) =  
N_{\R^n_+}(\ol{x}) = N_{\R_+}(\ol{x}_1) \times \dots \times N_{\R_+}(\ol{x}_n) = \{0\}^{|S|}\times \R^{|S^c|}_{-},
\end{equation}
since $S_u=\emptyset$, $S_{lu}=\{i\in {[n]} \colon \ol{x}_i\in (0,+\infty)\}=\supp(\ol{x})=:S$ and $S_l=S^c$.
Based on the considerations above we now derive the main result of this section.
\begin{theorem}[Dual Certificate] \label{thm:dual_certificate}
  A solution $\ol{x}$ to \eqref{eq:f-general-box} is unique, if and only if the following conditions hold
  \begin{itemize}
    \item[(i)] $\mc{N}(A) \cap \mc{N}(\row{D}{\Lambda}) \cap \mc{N}(\Psi) = \{0\},$
    \item[(ii)] $\exists \alpha \in \mb{R}^p,\exists \mu \in \mb{R}^n :$ \\
     $D^\top \alpha + \Psi \mu \in \mc{R}(A^\top), \alpha_{\Lambda^c} = \sign(\row{D}{\Lambda^c} \ol{x}), \|\alpha_{\Lambda}\|_\infty < 1, \mu > 0, $
  \end{itemize}
  where $\Lambda := \cosupp(D\ol{x})$ and $\Psi=\Psi(\ol{x})$ is a diagonal matrix that depends on $\ol{x}$ with entries
  \begin{align}
    \Psi_{ii} = \begin{cases}
      -1, &  \ol{x}_i = l_i, \\
      1, &  \ol{x}_i = u_i, \\
      0, & \text{otherwise.}
    \end{cases}
  \end{align}
  \begin{proof}
    By Lem.~\ref{lem:unique-minimizer-polyhedral} the uniqueness of  $\ol{x}$ as the unique solution of \eqref{eq:f-general-box} is equivalent to 
    \begin{equation}\label{eq:main-proof-1}
      0 \in \intr \left( \partial \left(  \|D\ol{x}\|_1 + \delta_{\mc{X}}(\ol{x}) + \delta_{{[l,u]}^n}(\ol{x}) \right)  \right).
    \end{equation}
    We rewrite \eqref{eq:main-proof-1} in the form of the two conditions
    \begin{itemize}
      \item[(a)] $\mb{R}^n = \aff\left( \partial \left(  \|D\ol{x}\|_1 + \partial \delta_{\mc{X}}(\ol{x}) + \partial \delta_{{[l,u]}}(\ol{x}) \right) \right) $,
      \item[(b)] $0 \in \rint \left( \partial \left(  \|D\ol{x}\|_1 + \partial \delta_{\mc{X}}(\ol{x}) + \partial \delta_{{[l,u]}}(\ol{x}) \right)  \right) $.
    \end{itemize}
  Condition (a) ensures that the interior of the subdifferential is not-empty, and can be recast as
  \begin{equation}\label{eq:Rn-decomp}
    \mb{R}^n = \aff( \partial   \|D\ol{x}\|_1) +  \aff(  \partial \delta_{\mc{X}}(\ol{x}) )+  \aff(  \partial \delta_{{[l,u]}}(\ol{x})).
  \end{equation}
By \eqref{eq:subdiff-Dl1} we get

  \begin{align}
   \aff (\partial \| D \ol{x} \|_1 ) = \ & \aff ( \{ D^\top \alpha \colon \alpha_{\Lambda^c} = \sign ( \row{D}{\Lambda^c} \ol{x} ) , \| \alpha_{\Lambda} \|_\infty \le 1  \}  )  \\
    = \ & \aff ( {\row{D}{\Lambda^c}}^\top \sign (\row{D}{\Lambda^c} \ol{x}  ) + \{ {\row{D}{\Lambda}}^\top \alpha \colon \|\alpha \|_\infty \le 1 \}   ) \\
    = \ & \underbrace{ {\row{D}{\Lambda^c}}^\top \sign (\row{D}{\Lambda^c} \ol{x} )}_{=:y_0} + \aff ( \{ \row{D}{\Lambda}^\top \alpha \colon \|\alpha \|_\infty \le 1 \}   ) \\
    = \ & y_0 + \mc{R}(  \row{D}{\Lambda}^\top ) . \label{eq:JanTodo_y0} 
  \end{align}
By \eqref{eq:subdiff-ind-X} we obtain
\begin{equation}
    \aff( \partial \delta_{\mc{X}} (\ol{x}) ) =  \aff( N_{\mc{X}}(\ol{x}) ) =   \aff ( \mc{N}(A)^\perp) = \mc{N}(A)^\perp.
  \end{equation}
And finally, by  \eqref{eq:def-normal-cone-box} we write

\begin{align}
  \aff( \partial \delta_{[l,u]} (\ol{x}) ) = \ & \aff( N_{[l,u]}(\ol{x}) ) = \aff( \R^{|S_l|}_- \times \{0\}^{|S_{lu}|}\times \R^{|S_u|}_{+})\\
  = \ &  \R^{|S_l|} \times \{0\}^{|S_{lu}|}\times \R^{|S_u|}      =  \mc{R}(\Psi).
\end{align}
Hence, \eqref{eq:Rn-decomp} is equivalent to $\R^n=\mc{R}(  \row{D}{\Lambda}^\top) + \mc{N}(A)^\perp +  \mc{R}(\Psi)$.
Taking the orthogonal, condition (a) now becomes

\begin{align}
  & \mc{R}(  \row{D}{\Lambda}^\top )^\bot \cap (\mc{N}(A)^\bot )^\bot \cap \mc{R}(\Psi)^\bot \\
  = \ & \mc{N}(\row{D}{\Lambda}) \cap \mc{N}(A) \cap \mc{N}(\Psi) = \{0\},
\end{align}
that yields (i). 

Further we reformulate condition (b). We use that for two convex sets $C_1,C_2$ it holds 

\begin{equation*}
\rint(C_1+C_2) = \rint(C_1) + \rint(C_2),
\end{equation*} 
see \eqref{eq:sum-relint}, and so we get

\begin{align}
  0 = & \rint ( \partial  \|D\ol{x}\|_1 + \partial \delta_{\mc{X}}(\ol{x}) + \partial \delta_{{[l,u]}}(\ol{x}) ) \\
  = \ & \rint ( \partial   \|D\ol{x}\|_1 )  + \rint ( \partial \delta_{\mc{X}}(\ol{x})) + \rint( \partial \delta_{{[l,u]}}(\ol{x}) )  \\
  = \ & \rint ( y_0 + \{ \row{D}{\Lambda}^\top \alpha \colon \|\alpha \|_\infty \le 1 \}  )  + \rint ( \mc{R}(A^\top) ) + \rint( \R^{|S_l|}_- \times \{0\}^{|S_{lu}|}\times \R^{|S_u|}_{+} ) \\
  = \ &  y_0 + \{   \row{D}{\Lambda}^\top \alpha \colon \|\alpha \|_\infty < 1\}   + \mc{R}(A^\top) +  \R^{|S_l|}_{--} \times \{0\}^{|S_{lu}|}\times \R^{|S_u|}_{++}\\
   = \ &  y_0 + \{   \row{D}{\Lambda}^\top \alpha \colon \|\alpha \|_\infty < 1\}   + \mc{R}(A^\top) + \{ \Psi \mu \colon \mu > 0 \}. 
\end{align}
  This shows (ii) and concludes the proof.
  \end{proof}
\end{theorem}

\begin{remark}
By Thm.~\ref{thm:dual_certificate} we can show that $\ol{x}$ is the \emph{unique} solution of \eqref{eq:f-general-box}
if and only if $(i)$ and $(ii)$ holds.  Condition $(i)$ is tested by evaluating whether 
$\bpm A^\top, & \row{D}{\Lambda}^\top, & \Psi^\top\epm^\top$ 
has full column rank. For verifying $(ii)$ we need
to test whether there exist $y\in \R^m$,  $\alpha \in \mb{R}^p$, $\mu \in \mb{R}^n :$
$D^\top \alpha + \Psi \mu = A^\top y, \alpha_{\Lambda^c} = \sign(\row{D}{\Lambda^c} \ol{x}), \|\alpha_{\Lambda}\|_\infty < 1, \mu > 0. $

The second condition $(ii)$ could be verified in practice by minimizing  $\|\alpha_{\Lambda}\|_\infty$ w.r.t. $y,\alpha,\mu$ while respecting the equality constraints
and converting the strict inequality constraint $\mu>0$ to $\mu>=\veps\eins$ for a small
$\veps$, e.g. $\veps=10^{-8}$, to ensure that the inequality is satisfied
strictly. This results in a linear program (LP),

\begin{align}
\min_{t,y,\alpha, \mu} t\qquad \text{s.t.}\quad  -t\eins& \le \alpha_{\Lambda} \le t\eins,\\
A^\top y&=D^\top \alpha + \Psi \mu \\
\mu &\ge \veps \eins \ . 
\end{align}
For the optimal solution $(t^\ast,y^\ast,\alpha^\ast,\mu^\ast)$  we have by definition the smallest possible $t:=\| \alpha_{\Lambda}\|_\infty$.
If $t^\ast$ is not smaller than one, then no $y$ exists
with smaller $t$. We therefore declare $\ol{x}$ the unique minimizer if $t^\ast<1$, and if $t^\ast\ge 1$, $\ol{x}$ cannot be the unique
minimizer. Numerically, we would test whether $t^\ast\le 1-\veps$. 
Technically, by applying the above procedure  we risk  rejecting a unique solution $\ol{x}$, 
for which $1-\veps<t^\ast<1$. Recall that the choice for $\veps$ is ad hoc.
\end{remark}

In order to resolve this issue we provide next a theoretically well-founded methodology to deal with the \emph{strict} feasibility problem

\begin{align}\label{eq:dual-certificate-1}
  & D^\top \alpha + \Psi \mu = A^\top y, \\ \label{eq:dual-certificate-2}
  & \alpha_{\Lambda^c} = \sign( \row{D}{\Lambda^c} \ol{x}), \\ \label{eq:dual-certificate-3}
  & \| \alpha_{\Lambda} \|_\infty < 1, \\ \label{eq:dual-certificate-4}
  & \mu > 0.
\end{align}
The above feasibility problem can be recast as a linear system of inequalities, as we will see next.
For this purpose we transform the above problem into the form $Mz=q,Pz < p$. 

\begin{theorem}\label{thm:check-strict-ineq}
Let $M$ be a matrix so that $\mc{N}(M^\top) = \{0\}$.
Then there is a  point $\ol{z}$ with $M\ol{z}=q,P\ol{z} < d$  
if and only if $v=0$ is the only feasible solution of the problem
  \begin{equation}
    q^\top u + p^\top v \le 0,\quad  M^\top u + P^\top v = 0,\quad  v \ge 0. \label{eq:jan_strict-ineq}
\end{equation} 
\end{theorem}
\begin{proof}
 We consider the following pair of linear programs \\
 \begin{minipage}{0.5\textwidth}
   \begin{align}
   \max \ & 0 \tag{P} \label{eq:jan_primalproblem} \\
   \text{s.t. } & Mz = q \notag \\
   & P z \le d \notag ,
   \end{align}
 \end{minipage}
\begin{minipage}{0.5\textwidth}
  \begin{align}
  \min \ & q^\top u + d^\top v \tag{DP} \label{eq:jan_dualproblem} \\
  \text{s.t. } & M^\top u + P^\top v = 0 \notag \\
  & v \ge 0 \notag
  \end{align}
\end{minipage}
where \eqref{eq:jan_primalproblem} is the primal problem and \eqref{eq:jan_dualproblem} its dual. 

First, note that under the  assumption  $\mc{N}(M^\top) = \{0\}$, vector $v=0$ is the only feasible solution of \eqref{eq:jan_strict-ineq} if and only if
$v=0$ is the only solution of \eqref{eq:jan_dualproblem}. Indeed, the common constraint $M^\top u + P^\top v = 0$ implies due to
the assumptions $v=0$ and $\mc{N}(M^\top) = \{0\}$ that $(u,v)=(0,0)$.

Hence, we can show the statement of the theorem by showing that  \eqref{eq:jan_primalproblem} is strictly feasible if and only if
$v=0$ is the only solution of \eqref{eq:jan_dualproblem}.

Assume on one hand that \eqref{eq:jan_primalproblem} is strictly feasible. Since any feasible solution is also a solution of
\eqref{eq:jan_primalproblem} we can deduce from the existence of a primal solution the existence of a dual solution.
The optimality conditions yield in particular $(d-P\ol{z})^\top v^\ast = 0$ for such a primal solution $\ol{z}$ of \eqref{eq:jan_primalproblem} 
and any dual solution $(u^\ast,v^\ast)$ of \eqref{eq:jan_dualproblem}.
Since $P\ol{z} < d$ holds, it follows that $v^\ast = 0$ is the only solution.

Assume on the other hand that for each dual solution $(u^\ast,v^\ast)$ it holds $v^\ast = 0$. In view of the strict complementarity condition \cite[Thm 10.7]{Vanderbei.2014} it exists a pair of primal and dual solutions such that 
\begin{equation}\label{eq:strict-complementarity}
(d-Pz^\ast) + v^\ast > 0.
\end{equation} 
Since by assumption $v^\ast = 0$, \eqref{eq:strict-complementarity} implies the existence of a primal solution $z^\ast$ with $Pz^\ast < d$.
\smallbreak
This concludes the proof. 
\end{proof}
\begin{remark} 
In view of Thm. ~\ref{thm:check-strict-ineq} we can check feasibility of the system \eqref{eq:dual-certificate-1} -- \eqref{eq:dual-certificate-4}
by linear programming. We first  associate to the  system \eqref{eq:dual-certificate-1} -- \eqref{eq:dual-certificate-4} with strict inequality constraints
the primal linear program
\begin{align*}
    \max_{\alpha,\mu,y} \ & 0 \tag{P$_{<}$} \label{eq:jan_prob3} \\
  \text{s.t. } & \underbrace{\bpm \row{D^\top}{\Lambda} & \Psi & A^\top \epm}_{=:M} \underbrace{ \bpm \alpha \\ \mu \\ y \epm }_{=:z} = \underbrace{-\row{D}{\Lambda^c}^\top\sign (\row{D}{\Lambda^c} \ol{x}) }_{=:q}\notag \\ 
  & \underbrace{ \bpm -I & 0 & 0 \\ I & 0 & 0 \\ 0 & -I & 0  \epm }_{=:P}\underbrace{\bpm \alpha \\ \mu \\ y \epm}_{=:z}  \le  \underbrace{\bpm \eins\\ \eins\\ 0 \epm}_{=:d}, \notag
\end{align*}
as in the Thm. ~\ref{thm:check-strict-ineq}. Assumption $\mc{N} (M^\top) = \{0\}$
in Thm. ~\ref{thm:check-strict-ineq} now corresponds to the first condition from Thm.~\ref{thm:dual_certificate}.  Assuming that it holds,  we can now 
verify   system \eqref{eq:dual-certificate-1} -- \eqref{eq:dual-certificate-4} and strict feasibility in
 \eqref{eq:jan_prob3} and if
\begin{align*}
\max \ & \langle \mb{1},u^+ + u^- + w \rangle \tag{DP$_{<}$} \label{eq:jan_prob4} \\
 \text{s.t. }   & \langle \mb{1},u^+ + u^-  \rangle \le  \langle \sign(\row{D}{\Lambda^c} \ol{x}), \row{D}{\Lambda^c} v \rangle  \\
 & \row{D}{\Lambda} v = u^+ - u^- \\
 & \Psi v = w \\
 & Av = 0 \\
 & u^+,u^-,w \ge 0
\end{align*}
has an optimal objective value equal to zero. 
\end{remark}
%
\subsection{Case Studies: \eqref{def-f1-l1}-\eqref{def-f6-TV-box}}
In the previous section we provided verifiable uniqueness conditions for problem \eqref{eq:f-general-box}.
In this section we specialize these conditions to verifiable uniqueness optimality conditions for our objective 
functions \eqref{def-f1-l1}-\eqref{def-f6-TV-box}. 
\begin{corollary}{(Case $f_1$)} \label{cor:f1-l1}
  A feasible point $\ol{x}$ of problem \eqref{eq:minf-exact-constr} with the objective function defined in \eqref{def-f1-l1} is the unique solution if and only if 
  \begin{center}
  $\begin{array}{c}
    \exists y \in \mb{R}^m:\col{A}{S}^\top y = \sign(\ol{x}_S) \wedge \|\col{A}{S^c}^\top y\|_\infty < 1 \text{ and} \\
    \col{A}{S} \text{ is injective,}
  \end{array}$
  \end{center}
  where $S:=\supp (\ol{x})$ and $\col{A}{S}$ selects the columns indexed by $S$.
  \end{corollary}
  \begin{proof}
  We derive the above claim from Thm. ~\ref{thm:dual_certificate}. In view of the definition in  \eqref{def-f1-l1}
  we set $D = I$ the identity matrix and $\Psi = 0$ the zero matrix. The latter is due to the fact that $\R^n$ can be seen as
  $\delta_{[l,u]}(\ol{x})$ with $l_i=-\infty$, $u_i=\infty$, for all $i\in{[n]}$. Hence $\ol{x}$ always lies in the interior of $\R^n$ and 
  $N_{\R^n}(\ol{x})=0$, compare \eqref{eq:def-normal-cone-box}.
  These choices specialize the conditions of Thm. \ref{thm:dual_certificate} to the following conditions
    \begin{itemize}
      \item[(i)] $\mc{N}(A) \cap \mc{N}(\row{I}{S^c}) = \{0\}$,
      \item[(ii)] $\exists \alpha \in \mb{R}^n : \alpha \in \mc{R}(A^\top), \alpha_S = \sign(\ol{x}_S), \|\alpha_{S^c}\|_\infty < 1$,
    \end{itemize}
    taking into account also that $\Lambda = S^c$ and $\mc{N}(\Psi) = \mb{R}^n$. 
    Condition (i) can be further simplified to
    
    \begin{align}
      \mc{N}(A) \cap \mc{N}(\row{I}{S^c}) =\mc{N}(A) \cap \mc{N}(\col{I}{S^c}) = \mc{N}(A) \cap \{ x \colon x_{S^c} = 0 \} =  \mc{N}(\col{A}{S}) = \{ 0 \},
    \end{align}
    which yields the condition that $\col{A}{S}$ has to be injective. We conclude the proof by transforming the condition (ii) to
    
    \begin{align}
     & \exists \alpha \in \mb{R}^n,y \in \mb{R}^m : \alpha = A^\top y \wedge \alpha_S = \sign(\ol{x}_S) \wedge \|\alpha_{S^c}\|_\infty < 1 \\
   \Leftrightarrow \  & \exists y \in \mb{R}^m   :  \col{A}{S}^\top y = \sign(\ol{x}_S) \wedge \|\col{A}{S^c}^\top y \|_\infty < 1.
    \end{align}
  \end{proof}

\begin{corollary}{(Case $f_2$)} \label{cor:f2-l1-nonneg}
  A feasible point $\ol{x}$ of problem \eqref{eq:minf-exact-constr} with the objective function defined in \eqref{def-f2-l1-nonneg}
 is the unique solution if and only if 
  \begin{center}
  $\begin{array}{c}
    \exists y \in \mb{R}^m:\col{A}{S}^\top y = \eins_{S} \wedge \col{A}{S^c}^\top y < \eins_{S^c}\text{ and} \\
    \col{A}{S} \text{ is injective,}
  \end{array}$
  \end{center}
  with $S:=\supp (\ol{x})$.
  \end{corollary}
  \begin{proof}  As in the proof of Cor. ~\ref{cor:f1-l1} we apply Thm. ~\ref{thm:dual_certificate} with $D=I$ and
  $\Psi_{ii}=0$, for $i\in S$ and $\Psi_{ii}=-1$, for $i\in S^c$, since $N_{R^n_+}(\ol{x})=\{0\}^{|S|}\times \R_-^{|S^c|}$.
  Conditions (i) and (ii) in Thm. \ref{thm:dual_certificate} become
    \begin{itemize}
      \item[(i)] $\mc{N}(A) \cap \mc{N}(\row{I}{S^c}) \cap \mc{N}(\Psi) = \{0\}$,
      \item[(ii)] $\exists \alpha \in \mb{R}^n,\exists \mu \in \mb{R}^n :$ \\
      $ \alpha +\Psi \mu \in \mc{R}(A^\top), \alpha_S = \sign(\ol{x}_S), \|\alpha_{S^c}\|_\infty < 1, \mu >0$.  
   \end{itemize}
   Since $\mc{N}(\row{I}{S^c}) = \mc{N}(\Psi) $ condition (i) above simplifies as in the proof of Cor. ~\ref{cor:f1-l1} to
    condition: $\col{A}{S}$ is injective.
    
    Since $\{\Psi \mu \colon \mu \in\R^n_{++}\}=\Psi R^n_{++} = \{0\}^{|S|}\times \R^{|S^c|}_{--}$ and
    \begin{equation*}
    \{\alpha_{S^c} \colon  \|\alpha_{S^c} \|_\infty < 1\} + \R^{|S^c|}_{--} = \eins_{S^{c}} + \R^{|S^c|}_{--},
    \end{equation*}
    condition (ii) above is equivalent to the existence of a vector ${y}\in  {\R}^{m}$ such that 
\begin{equation*}
A^{\top} {y}\in  \eins + \{0\}^{|S|}\times \R^{|S^c|}_{--} ,
\end{equation*}
that can be written as
\begin{equation}
\col{A}{S}^\top y =\eins_S, \quad \col{A}{S^{c}}^\top y <\eins_{S^{c}}.
\end{equation}
This completes the proof.
  \end{proof}
\begin{corollary}{(Case $f_3$)} \label{cor:f3-l1-box}
 A binary feasible point $\ol{x}\in\{0,1\}^n$ of problem \eqref{eq:minf-exact-constr} with the objective function defined in  \eqref{def-f3-l1-box} is the unique solution if and only if
 
  \begin{align}
    \exists y \in \mb{R}^m \ : \ \col{A}{S}^\top y > \mb{1} \wedge  \col{A}{S^c}^\top y < \mb{1},
  \end{align}
  with $S := \supp (\ol{x})$.
  \begin{proof}
  As above we specialize Thm. ~\ref{thm:dual_certificate} with $D=I$ and 
  $\Psi_{ii}=-1$, for $i\in S^c$ and $\Psi_{ii}=1$, for $i\in S$, since $N_{[0,1]^n}(\ol{x})=\R^{|S|}_{+}\times R_{-}^{|S^c|}$ for
  $\ol{x}\in\{0,1\}^n$.
  Hence $\mc{N}(\Psi) = \{0\}$. As a consequence condition (i) Thm. ~\ref{thm:dual_certificate}  is automatically fulfilled. 
  We now focus on condition (ii) that now reads
  
    \begin{align} \label{eq:cor_f3-l1-box}
      \begin{pmatrix}
        \alpha_S \\
        \alpha_{S^c}
      \end{pmatrix} +
      \begin{pmatrix}
        \mu^+ \\
        \mu^-
      \end{pmatrix}
      \in \mc{R}\left( \begin{pmatrix} \col{A}{S}^\top \\ \col{A}{S^c}^\top \end{pmatrix} \right)   , \text{ with } \mu^- < 0, \mu^+ > 0, \alpha_{S}=\mb{1}, \|\alpha_{S^c} \|_\infty < 1,
    \end{align}
    since for  $\ol{x} \in \{0,1\}^n$ we have $\alpha_{S} = \sign(\ol{x}_S) = \mb{1}$. In view of
    $$
    \{\alpha_{S^c} \colon  \|\alpha_{S^c} \|_\infty < 1\} + \R^{|S^c|}_{--} = \eins_{S^{c}} + \R^{|S^c|}_{--},
    $$
    condition \eqref{eq:cor_f3-l1-box} 
    is equivalent to the existence of a vector ${y}\in  {\R}^{m}$ such that 
$$A^{\top} {y}\in  \eins + \R^{|S|}_{++}\times \R^{|S^c|}_{--} ,$$ 
that can be written as
\begin{equation}
\col{A}{S}^\top y > \eins_S, \quad \col{A}{S^{c}}^\top y <\eins_{S^{c}}.
\end{equation}
\end{proof}
\end{corollary}
\begin{corollary}{(Case $f_4$)} \label{cor:f4-TV}
 A feasible point $\ol{x}\in\R^n$ of problem \eqref{eq:minf-exact-constr} with the objective function defined in  \eqref{def-f4-TV} 
 is the unique solution if and only if 
 \begin{itemize}
    \item[(i)] $\mc{N}(A) \cap \mc{N}(\row{D}{\Lambda}) = \{0\},$
    \item[(ii)] $\exists \alpha : D^\top\alpha \in \mc{R}(A^\top ), \alpha_{\Lambda^c} = \sign(\row{D}{\Lambda^c}\ol{x}), \|\alpha_{\Lambda} \|_\infty < 1,$
  \end{itemize}
  where $\Lambda = \cosupp(D \ol{x})$.
  \end{corollary}
  \begin{proof} As in the proof of Cor. \ref{cor:f1-l1} we conclude that $\Psi = 0$ is the zero matrix. 
  Hence, the two conditions of  Thm. ~\ref{thm:dual_certificate}  simplify to the two conditions above, that do not involve the variable $\mu$.
  \end{proof}

\begin{corollary}{(Case $f_5$)} \label{cor:f5-TV-nonneg}
A feasible point $\ol{x}\in\R^n_+$ of problem \eqref{eq:minf-exact-constr} with the objective function defined in \eqref{def-f5-TV-nonneg} 
is the unique solution if and only if 
\begin{itemize}
    \item[(i)] $\mc{N}(A) \cap \mc{N}(\row{D}{\Lambda}) \cap \mc{N}( \Psi ) =  \{0\},$
    \item[(ii)] $\exists \alpha,  \exists \mu: D^\top\alpha + \Psi \mu \in \mc{R}(A^\top ), \alpha = \sign(\row{D}{\Lambda^c}\ol{x}), \|\alpha_{\Lambda} \|_\infty < 1,\mu > 0$
  \end{itemize}
  where $\Lambda = \cosupp(D \ol{x})$ and
  
  \begin{align}
    \Psi_{ii} = \begin{cases}
      -1, &  \ol{x}_i = 0, \\
      0, & \text{otherwise.}
    \end{cases}
  \end{align}
  \end{corollary}
  \begin{proof} 
  Immediate from Thm. ~\ref{thm:dual_certificate} and the definition of $f_5$ in \eqref{def-f5-TV-nonneg}  and $\Psi$.
  \end{proof}

\begin{corollary}{(Case $f_6$)} \label{cor:f6-TV-box}
A binary feasible point $\ol{x}\in\{0,1\}^n$ of problem \eqref{eq:minf-exact-constr} with the objective function defined in \eqref{def-f6-TV-box} 
is the unique solution if and only if 
\begin{itemize}
    \item[(i)] $\mc{N}(A) \cap \mc{N}(\row{D}{\Lambda}) \cap \mc{N}( \Psi ) =  \{0\},$
    \item[(ii)] $\exists \alpha,  \exists \mu: D^\top\alpha + \Psi \mu \in \mc{R}(A^\top ), \alpha = \sign(\row{D}{\Lambda^c}\ol{x}), \|\alpha_{\Lambda} \|_\infty < 1,\mu > 0$
  \end{itemize}
  where $\Lambda = \cosupp(D \ol{x})$ and
  
  \begin{align}
    \Psi_{ii} = \begin{cases}
      -1, &  \ol{x}_i = 0, \\
       1, &  \ol{x}_i = 1.
    \end{cases}
  \end{align}
  \end{corollary}
  \begin{proof} 
  In view of the definition of $f_6$ in \eqref{def-f6-TV-box}  and $\Psi$ above
   $\mc{N}(\Psi) = \{0\}$. As a consequence condition (i) Thm. ~\ref{thm:dual_certificate}  is automatically fulfilled. 
   Condition (ii) from Thm. ~\ref{thm:dual_certificate} is kept unchanged.
  \end{proof}

\section{Probabilistic Uniqueness}\label{sec:prob-unique}

In this section we are concerned with probabilistic solution uniqueness 
when minimizing some structure enforcing regularizer $f:\R^n\to\ol{\R}$
via \eqref{eq:minf-exact-constr} in the noiseless setting of \eqref{eq:Ax=b}.
In particular we address the question of predicting uniqueness with high probability
for a specific number of random linear measurements,
when knowing \emph{only} the solution co-/sparsity.

For $f=\|\cdot\|_1$ and  a Gaussian matrix $A$ it is well-known that the recovery of
a sparse vector $x$ depends only on its sparsity level, and is independent of the locations or values of the
nonzero entries. There are precise relations between the sparsity $s$, ambient dimension $n$, and
the number of samples $m$ that guarantees success of \eqref{eq:minf-exact-constr}.
Moreover, several authors have shown that there is a transition from absolute success to absolute failure, and they 
have \emph{accurately} characterized the location of the threshold, also called
\emph{phase transition}, in the $\ell_1$-case and other \emph{norms}, e.g. nuclear norm.
For the $\TV$-seminorm such a complete analysis is still missing.
As mentioned in the introduction, the authors in  \cite{Zhang-1DTV} have recently shown 
that in the case of 1D TV-minimization, case \eqref{def-f4-TV},
the phase transition can be accurately described by an effective bound
for the statistical dimension of a descent cone, which is based on the squared distance of a standard normal vector to the 
subdifferential of the objective function at the sought solution $\ol{x}$,

\begin{equation*}
\min_{\tau \ge 0} \EE [\dist(X,\tau \partial f(\ol{x})].
\end{equation*}

We will explore next if the same relation that describes phase transitions for $\ell_1$-minimization and  1D $\TV$-minimization
also hold for all our objective functions in \eqref{def-f1-l1}-\eqref{def-f6-TV-box}.

\subsection{Phase Transitions in Random Linear Inverse Problems} 

We collect here some recent results of convex signal reconstruction
with a Gaussian sampling model and briefly explain how classical results for Gaussian processes lead to a sharp bound for the
number of Gaussian measurements that suffice for exact recovery.

\begin{definition}{(Statistical dimension, \cite{Amelunxen2014living})}
 The \emph{statistical dimension} $\delta(K)$ of a closed, convex cone $K\subset\R^n$ is the quantity
\begin{equation}\label{eq:def-stat-dim}
\delta(K)=\EE[\Vert\Pi_K(X)\Vert_2^2],
\end{equation}
where $\Pi_K$ is the Euclidean projection onto $K$ and $X$ is a standard normal vector, i.e. $X\sim\mc{N}(0,I_n)$.
 \end{definition}

\begin{definition}{(Descent cone)}
The \emph{descent cone} of a proper convex function $f:\R^n\to\ol{\R}$ at a point $x$ is
\begin{equation}\label{eq:def-descent-cone}
\mc{D}_f(\ol{x}) := \cone \{ z-\ol{x} \colon f(z)\leq f(\ol{x})\}, 
\end{equation}
i.e. the conic hull of the directions that do not increase $f$ near $\ol{x}$.
 \end{definition}

\begin{definition}{(Gaussian width, \cite{Chandrasekaran2012convex})}
The \emph{Gaussian width} of $C\subset \R^n$ is

	\begin{equation*}
		\omega (C) = \EE [\sup_{z\in C} \la X,z \ra ],
	\end{equation*}
	where $X$ is a standard normal vector, i.e. $X\sim\mc{N}(0,I_n)$.
\end{definition}
	
\begin{remark} 
One has, see \cite{Amelunxen2014living}, the relationship
\begin{equation}
\omega(K)^2 \le \delta(K) \le \omega(K)^2 + 1.
 \end{equation}
\end{remark}

A classic result for solution uniqueness, equivalent to the exact recovery of individual vectors
\cite[Thm. 4.35]{Rauhut-book}, is given next.
\begin{proposition}\label{prop:unique_cond_intersection}
The vector $\ol{x}$ is the \emph{unique} solution of the convex program \eqref{eq:minf-exact-constr} if and only if 
\begin{equation}\label{eq:cone-intersection}
\mc{D}_f(\ol{x}) \cap \mc{N}(A) = \{0\}.
\end{equation}
\end{proposition}

Condition \eqref{eq:cone-intersection} holds with high probability
for a random matrix $A$ is equal to the probability that a randomly rotated convex cone, here $\mc{N}(A)$,  
shares a ray with a fixed convex cone $\mc{D}_f(\ol{x})$. This probability is bounded in terms of the statistical dimension $\delta(\mc{D}_f(\ol{x}))$ as stated next.

\begin{theorem}{\cite[Thm. II]{Amelunxen2014living}} Fix a tolerance parameter $\veps \in(0,1)$. 
Let $\ol{x} \in\R^n$ be a fixed vector, and let $f:\R^n\to\ol{\R}$ be a proper convex function. Suppose $A\in\R^{m\times n}$ has
independent standard normal entries, and let $b=A\ol{x}$. Then
\begin{itemize}
\item $m \le \delta(\mc{D}_f(\ol{x})) - a_n \sqrt{n}$ $\Longrightarrow$ \eqref{eq:minf-exact-constr} succeeds with probability $\le \veps$;
\item $m \ge \delta(\mc{D}_f(\ol{x})) + a_n \sqrt{n}$ $\Longrightarrow$ \eqref{eq:minf-exact-constr} succeeds with probability $\ge 1- \veps$,
\end{itemize}
with $a_n:=\sqrt{8\log(4/\veps)}$.
\end{theorem}

There is a useful tool for bounding the  statistical dimension in terms of $\partial f(\ol{x})$. Interestingly, this bound
is tight for some classes of $f$, e.g. norms.
\begin{proposition}{\cite{Chandrasekaran2012convex,Amelunxen2014living}} \label{prop:Jtau}
Let  $f:\R^n\to\ol{\R}$ be a proper convex function and let $x\in\R^n$.
Assume that the subdifferential  $\partial f(\ol{x})$ is nonempty, compact, and does not contain the origin. Define the function
 \begin{equation}\label{eq:def-J-tau}
 J(\tau)=J(\tau;\partial f(\ol{x})):=\EE [\dist(X,\tau \partial f(\ol{x})], \quad {\rm for} \ \tau\ge 0,
 \end{equation}
 where $X\sim\mc{N}(0,I)$. We have
 \begin{equation}\label{eq:def-J-tau}
 \delta(\mc{D}_f(\ol{x}))\le \inf_{\tau\ge 0} J(\tau) .
 \end{equation}
Furthermore, the function $J$ is strictly convex, continuous at $\tau =0 $, and differentiable for $\tau \ge 0$. It achieves its
minimum at a unique point.
\end{proposition} 
The above results provide a recipe to upper bound $ \delta(\mc{D}_f(\ol{x}))$. This consists of the following steps.
\begin{itemize}
\item Compute $\partial f(\ol{x})$;
\item For each $\tau> 0$, compute $J(\tau)=\EE [\dist(X,\tau \partial f(\ol{x})]$;
\item Find the unique solution to $J'(\tau)=0$.
\end{itemize}

We also have the following error bound.
\begin{theorem}{\cite[Thm. 4.3.]{Amelunxen2014living}}\label{thm:norm-error-bound}
Let $f:\R^n\to\R$ be a norm on $\R^n$ and let $x\in\R^n\setminus\{0\}$. Then
\begin{equation}\label{eq:error-bound}
0\le \inf_{\tau\ge 0} J(\tau)  - \delta(\mc{D}_f(\ol{x})) \le \frac{2\sup\{\Vert p\Vert\colon p\in\partial f(\ol{x})\}}{f(\ol{x}/\Vert \ol{x}\Vert)}
\end{equation}
\end{theorem}
For the $\ell_1$-case the r.h.s. can be made very accurate for large sparsity parameters $s$
providing an accurate estimate of the statistical dimension of the $\ell_1$-descent cone.
On the other hand, this error estimate is not very accurate when the sparsity $s$ is small.
The work in \cite{Foygel2014} contains a bound that improves this result and also extends to more general functions
with some additional properties.
\begin{theorem}{\cite[Prop.~1]{Foygel2014}}\label{thm:Foygel-bound} Suppose that, for $\ol{x}\ne 0$, the subdifferential $\partial f(\ol{x})$ satisfies
a weak decomposability assumption, i.e.
\begin{equation}\label{eq:weak-decomposable-cond}
\exists \ol{p} \in \partial f(\ol{x}) \quad \text{such that}\quad \la p - \ol{p},\ol{p} \ra =0, \qquad \forall p \in \partial f(\ol{x}).
\end{equation}
Then
\begin{equation}\label{eq:Foygel-bound}
 \delta(\mc{D}_f(\ol{x}))\le \inf_{\tau\ge 0} J(\tau)\le \delta(\mc{D}_f(\ol{x})) +6 ,
 \end{equation}
with $J(\tau)= \EE [\dist(X,\tau \partial f(\ol{x})]$.
\end{theorem}

\subsection{Phase Transitions for Case Studies: \eqref{def-f1-l1}-\eqref{def-f6-TV-box}}
Above we discussed probabilistic uniqueness  for the general problem $$\min_{Ax=A\ol{x}} f(x).$$
In this section we discuss how and if we can specialize these results to predict
uniqueness for our objective functions \eqref{def-f1-l1}-\eqref{def-f6-TV-box}. 

\begin{remark}
In the case of $f_1(x) = \|x\|_1$ from \eqref{def-f1-l1}
and $f_4(x) = \|D x\|_1$ from  \eqref{def-f4-TV} the statistical
dimension of the descent cone  $\mc{D}_f(\ol{x})$ can be well approximated 
in terms of the expected squared distance to the scaled subdifferential. 
The reason is that
the subdifferential $\partial f(\ol{x})$ is compact, see \eqref{eq:subdiff-Dl1}, and thus by
Prop. \ref{prop:Jtau} we know that $J(\cdot):\R_+\to \R_+$ has a unique minimizer.
Moreover, $\min_{\tau\ge 0 }J(\tau)$ provides a tight approximation in view of
Thm. \ref{thm:norm-error-bound}  and Thm. \ref{thm:Foygel-bound}, since $f_1$ is a norm
and $\partial f_4$ is weakly decomposable for the choice $D$ equal to the 1D finite difference operator, 
as shown in \cite{Zhang-1DTV}. For the 2D finite difference operator the weakly decomposable property has not been
previously  investigated.
\end{remark}

\begin{remark}\label{rem:problematic-choices} The above situation changes for the objectives
  
\begin{align*}
f_2(x) &= \|x\|_1  + \delta_{\R^n_+}(x) , \qquad f_3(x) = \|x\|_1  + \delta_{{[0,1]}^n}(x) \\
f_5(x) &= \|Dx\|_1  + \delta_{\R^n_+}(x) , \qquad
f_6(x) = \|Dx\|_1  + \delta_{{[0,1]}^n}(x),  
\end{align*}
from  \eqref{def-f2-l1-nonneg}, \eqref{def-f3-l1-box}, \eqref{def-f5-TV-nonneg} and \eqref{def-f6-TV-box}.
It is clear that $\partial f(\ol{x})$ is \emph{not} compact, for all above choices, compare \eqref{eq:subdiff-Dl1} and \eqref{eq:subdiff-ind-box}.
Thus, we have no guarantee that, $J$ has even a minimizer.
Since $\partial f(\ol{x})$ is \emph{not} weakly decomposable, even if  $\min_{\tau\ge 0 }J(\tau)$ exists, we are not guaranteed  
a tight upper bound on the statistical dimension $ \delta(\mc{D}_f(\ol{x}))$ 
and consequently on the number of sufficient Gaussian measurements $m$ for exact recovery of $\ol{x}$ with high probability. 
\end{remark}

Explicit curves for $\min_{\tau}J(\tau)$ can be computed only in the case of $f_1$, $f_2$, $f_3$ along the lines of \cite{Amelunxen2014living}.
We skip the details here and illustrate this curves in Sect.~\ref{sec:experiments}.
For $f_4$, $f_5$ and $f_6$ explicit curves for $J$  and for $\min_{\tau}J(\tau)$  are missing.
We use an approximation to the upper bound $J(\tau) = \mathbb{E}\left[\dist^2 (X,\tau \partial f(\ol{x})) \right]$ where $X \sim \mc{N}(0,I)$. To approximate the expected value we use a very large $k \in\N$ and calculate

\begin{equation*}
  h_k(\tau) = \frac 1k \sum_{i=1}^k \dist^2 (X_i,\tau \partial f(\ol{x})) =: J_{\rm approx}(\tau) \approx J(\tau),
\end{equation*}
where each $X_i$ is sampled from the normal distribution $\mc{N}(0,I)$.
Note that for computing each $\dist^2 (X_i,\tau \partial f(\ol{x}))$ one is faced with solving a quadratic program, i.e.

\begin{equation*}
\min \|X_i - \tau y\|^2_2\qquad \text{subject to} \quad y \in \partial f(\ol{x}).
\end{equation*}
Note that the constraints are linear since $\partial f(\ol{x})$ is described by linear equalities and inequalities, compare \eqref{eq:subdiff-Dl1} and \eqref{eq:subdiff-ind-box}.

Based on the empirical results from Sect. \ref{sec:experiments} we conjecture.
\begin{conjecture}
For the choices of $f$ in Remark \ref{rem:problematic-choices} the statistical dimension is well approximated similar to \eqref{eq:Foygel-bound}
by the squared Euclidean distance to the scaled subdifferential, i.e.

\begin{equation*}
\min_\tau J(\tau) := \mathbb{E}\left[\dist^2 (X,\tau \partial f(\ol{x})) \right] \approx \delta(\partial f(\ol{x})).
\end{equation*}
\end{conjecture}

\section{Experiments}\label{sec:experiments}

Our objective in this section is to  \emph{empirically}  verify whether
phase transitions occur at
\begin{equation*}
\min_{\tau\ge 0 }J(\tau),
\end{equation*}
which is a \emph{guaranteed} upper bound to the statistical dimension $\delta(\mc{D}_f(\ol{x}))$ \emph{only} in 
some of the considered cases.
We note further, that precise expressions for $J(\tau)$ are only available for $\ell_1$-minimization 
and cases \eqref{def-f1-l1}--\eqref{def-f3-l1-box}. 
For the $\TV$-minimization and  cases \eqref{def-f4-TV}--\eqref{def-f6-TV-box} we
construct phase diagrams by \emph{numerically estimating} $\min_{\tau\ge 0 }J(\tau)$ and comparing  the resulting bounds
to the average-case results for both solving the reconstruction problems \eqref{eq:minf-exact-constr} 
and certifying uniqueness  by the procedure derived in  Sect.~\ref{sec:individual-unique}.

\subsection{1D Empirical Phase Transitions and Theoretical Bounds}
In this section we address  $\ell_1$-minimization and one dimensional $\TV$-minimization without and with constraints. The 1D TV-regularizer reads
$\|\partial_n x\|_1$,   where $\partial_n \in \R^{n-1 \times n}$ is the one-dimensional discrete derivative operator. Applying $D=\partial _n$ to a signal $x$ gives the 
the entries' differences for consecutive indices

\begin{align}
  Dx = \partial_n x = \begin{pmatrix}
      x_1 - x_0 \\ \vdots \\ x_n - x_{n-1}
  \end{pmatrix}.
\end{align}
\begin{figure}
 \includegraphics[width=\textwidth]{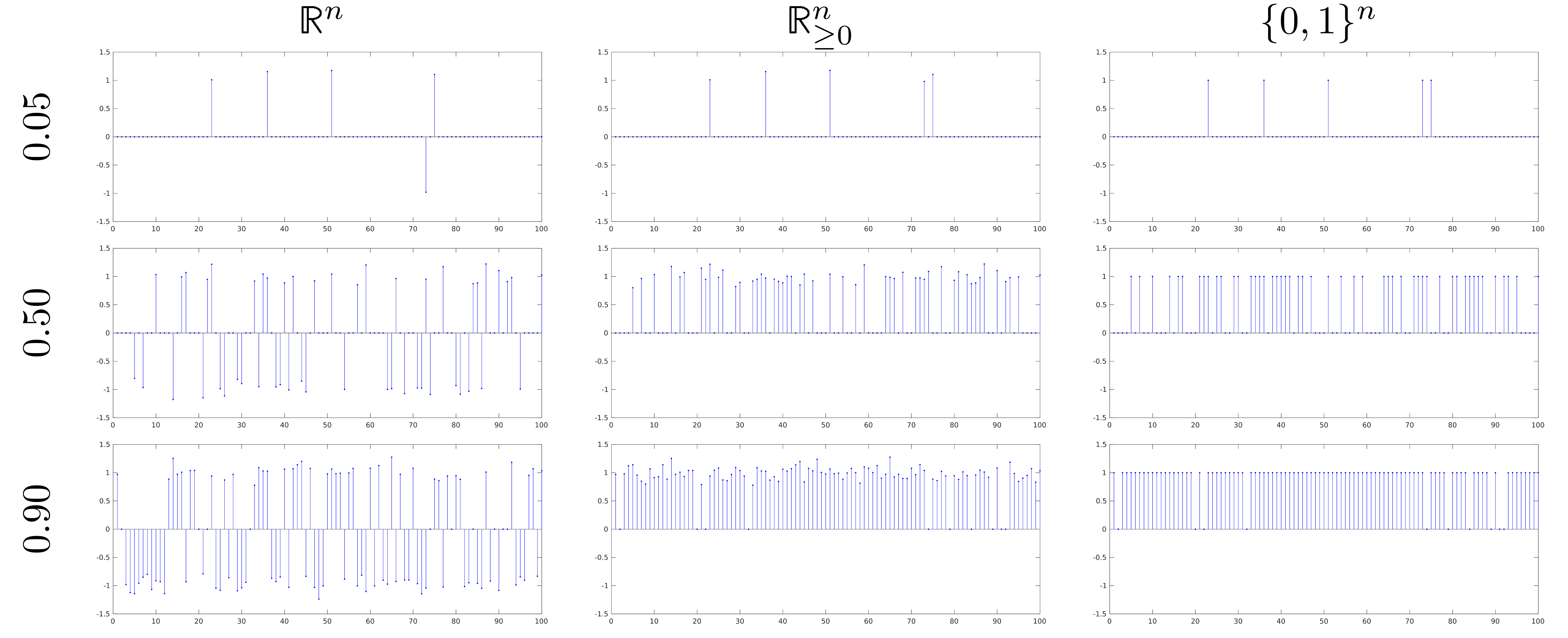}
  \caption{\emph{Testset for 1D sparse signals}. The rows display signals $\ol{x}$ with equal number of non-zeros in $\ol{x}$, 
  i.e. $\|\ol{x}\|_0$.
  The columns display signals that have entries sampled either from $\R, \R_+$ or $\{0,1\}$
  from left to right. 
  \label{fig:1DSTestset}}
\end{figure}
\begin{figure}
  \includegraphics[width=\textwidth]{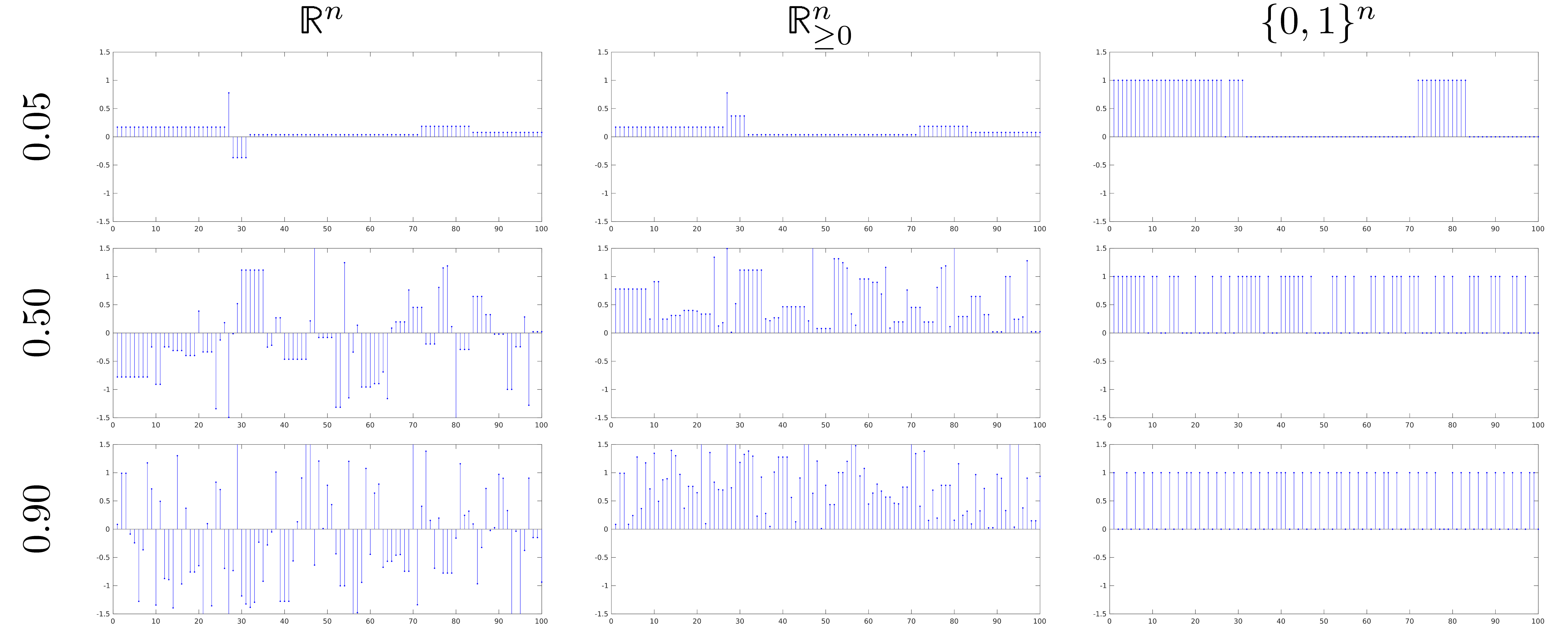}
  \caption{\emph{Testset for 1D gradient sparse signals}. The rows display signals $\ol{x}$ with equal number of non-zeros in $D\ol{x}$, 
  i.e. $\|D\ol{x}\|_0$,  with respect to the 1D finite difference operator.
  The columns display signals that have entries sampled either from $\R, \R_+$ or $\{0,1\}$
  from left to right.  
  \label{fig:21DTTestset}}
\end{figure}
\subsubsection{Testset}
Our testset consists of several randomly generated signals with specified  relative sparsity $\rho$ ranging from $0.05$ to $0.95$ with step size $0.05$. For each of these relative sparsities we create a signal so that 

\begin{align}
  \rho \approx \frac{\|\ol{x}\|_0}{n} \qquad \text{ resp. } \qquad  \rho \approx \frac{\|D\ol{x}\|_0}{p}.
\end{align} 
Hence, we cover almost the full range from highly sparse to dense signals. In addition, we create three different classes of signals: with real-valued, non-negative or binary entries, i.e. $x_i\in\R$, $x_i\ge 0$ are $x_i \in \{0,1\}$ for all $i\in{[n]}$. 

Creating random  real-valued, non-negative or binary signals of a given sparsity is immediate.
To create signals that are sparse in a transformed domain we use an idea from of \cite{Nam2013}. Instead of choosing the support of
$\ol{x}$ uniformly at random, ones chooses a set $\Lambda \subseteq {[p]}$ in order to create a signal with $\row{D}{\Lambda} \ol{x} = 0$. Having the subset $\Lambda$ and a randomly created vector $v$ with non-zero entries we obtain the desired signal by calculating 

\begin{align}
\ol{x} = (I - \row{D}{\Lambda}^\top(\row{D}{\Lambda} \row{D}{\Lambda}^\top)^{-1}\row{D}{\Lambda})v.
\end{align}
Since $D\ol{x}$ represents differences of signal entries for consecutive indices, taking the absolute value of each  signal entry will most likely not change the cosupport of the resulting signal. Therefore, we can use this method to generate non-negative signals. 
Binary signals can be obtained by using Algorithm \ref{alg:convertsignal}.

\begin{algorithm}[t]
  \begin{tabular}{ll}
    \textbf{Input}: & $\Lambda \subseteq [n-1]$\\
    \textbf{Output}: & Binary signal $\ol{x} \in \{0,1\}^n$
  \end{tabular}
  \hrule
  $\ol{x}_0 = 1$\;
  \For{$i \in [n-1]$}{
    \eIf{$i\in \Lambda$}{
      $\ol{x}_{i+1} \leftarrow  \ol{x}_i$\; 
    }{
      $\ol{x}_{i+1} \leftarrow  1-\ol{x}_i$\; 
    }
  }
  \KwRet{$\ol{x}$;}
  \caption{Algorithm for converting a random gradient sparse signal (w.r.t. the 1D finite differences operator) into a binary gradient 
  sparse signal. 
  \label{alg:convertsignal}}
\end{algorithm}

\subsubsection{$\ell_1$-Minimization}
In this section we compare the existing tight upper bounds for the statistical dimension $\delta(\mc{D}_f(\ol{x}))$
for $f$ defined in \eqref{def-f1-l1}--\eqref{def-f3-l1-box} to the empirical phase transition obtained for the above set of signals. 

To this end we first set the ambient dimension $n=100$ and generate 10 instances of a sparse 1D signal, see Figure \ref{fig:1DSTestset}, and a random Gaussian matrix $A\in\R^{m\times }$. For each pair of relative sparsity and given number $m$ of random linear measurements we verify if the signal is the unique 
solution of  \eqref{eq:minf-exact-constr} with $f$ defined in \eqref{def-f1-l1}--\eqref{def-f3-l1-box}.
Uniqueness is tested as detailed in Sect. \ref{sec:individual-unique}.
We used Mosek\footnote{\url{https://www.mosek.com/}} to solve the optimization problems.
The gray value plots in Figure \ref{fig:phase_transitions}, first row, show the empirical probability that a given 1D sparse signal
 is \emph{uniquely} reconstructed by the convex relaxation. White means recovered and unique, and black non recovered. The red 
 curves show $\min_{\tau\ge 0 }J(\tau)$, plotted with Mathematica\footnote{\url{https://www.wolfram.com/mathematica/}}, and separate these regions accurately. 

\subsubsection{$\TV$-Minimization}
To generate the phase diagrams we proceed as above but use the gradient sparse 1D test signals, see Figure \ref{fig:21DTTestset}.
Since explicit curves for $J$ are missing we use an approximation to the upper bound $J(\tau) = \mathbb{E}\left[\dist^2 (X,\tau \partial f(\ol{x})) \right]$, where $X \sim \mc{N}(0,I)$. To approximate the expected value we use a large $k =10000$ and calculate

\begin{equation*}
  h_k(\tau) = \frac 1k \sum_{i=1}^k \dist^2 (X_i,\tau \partial f(\ol{x})) =: J_{\rm approx}(\tau) \approx J(\tau),
\end{equation*}
where each $X_i$ is sampled from the normal distribution $\mc{N}(0,I)$. Additionally we know from \cite[Lem C.1]{Amelunxen2014living} that the minimum is unique and lies in the interval $\left[0,\frac{2\|\ol{x}\|_2}{b}\right]$ with $b \le \|w\|_2$ for all $w \in \partial f(\ol{x})$, provided
$\partial f(\ol{x})$ is bounded. This latter condition is only satisfied for $f_4$, but not for $f_5$ and $f_6$. Nevertheless, we successfully minimize numerically the
univariate function $J_{\rm approx}$ above by using the BiSection~\cite{Faires1985} method.
Note that for computing each $\dist^2 (X_i,\tau \partial f(\ol{x}))$ one is faced with solving a quadratic program. Note that
$\partial f(\ol{x})$ is described by linear equalities and inequalities, compare \eqref{eq:subdiff-Dl1} and \eqref{eq:subdiff-ind-box}.
The empirical probability of uniqueness  shown as gray value plots in Figure \ref{fig:phase_transitions}, middle row,
is plotted along with the curves describing   $\min_{\tau\ge 0} J_{\rm approx}(\tau)$.
We emphasize the perfect agreement of the phase transition.

\subsection{2D Empirical Phase Transitions and Theoretical Bounds}
\begin{figure}
 \includegraphics[width=\textwidth]{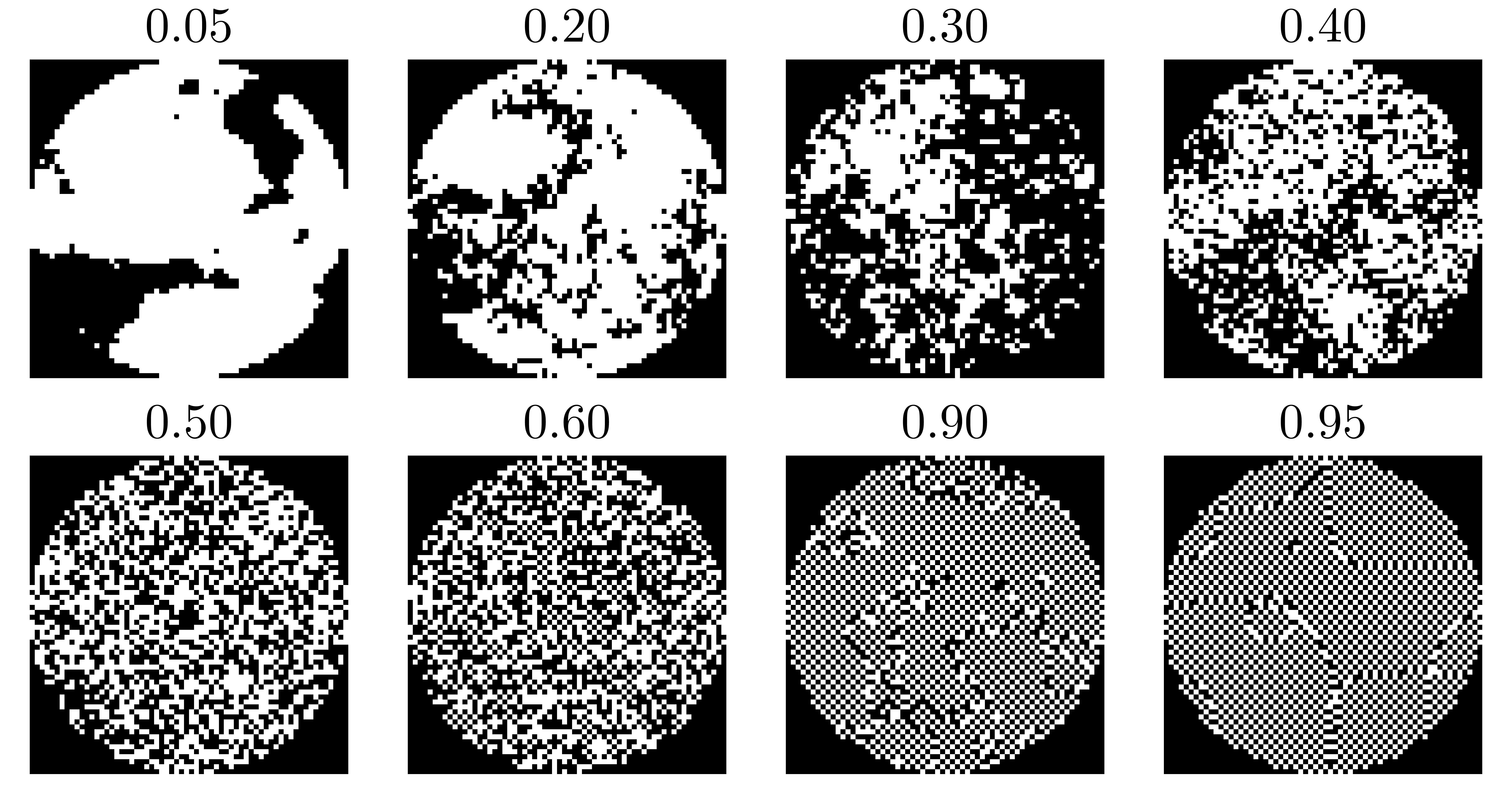}
  \caption{
  \emph{Testset for 2D gradient sparse images}. Each image label shows the relative sparsity of $D\ol{x}$, 
  i.e. $\|D\ol{x}\|_0/p$,  with respect to the 2D finite difference operator. Our test images are $64\times 64$.
  \label{fig:2DTestset}}
\end{figure}
\subsubsection{Testset}
 Creating gradient sparse images with a given gradient sparsity is more involved than in the one dimensional case. 
 Using random support sets $\Lambda \subseteq [p]$ with $|\Lambda| \ge n$ and the projecting technique used in the 1D case,
 one would most likely obtain constant 2D images. Thus we use a different approach to construct random gradient sparse images. 
 To this end, we randomly add binary images with homogeneous areas and use the modulo operation to binarize again the result.
 We show some results in Figure \ref{fig:2DTestset}. Since it is easy to identify all connected components in a binary image, we can assign random real values to the different connected components. The resulting discrete gradient of the new image has the same number of non-zeros, but the image is \emph{non binary}.
\begin{figure}
  \centering
 \includegraphics[width=\textwidth]{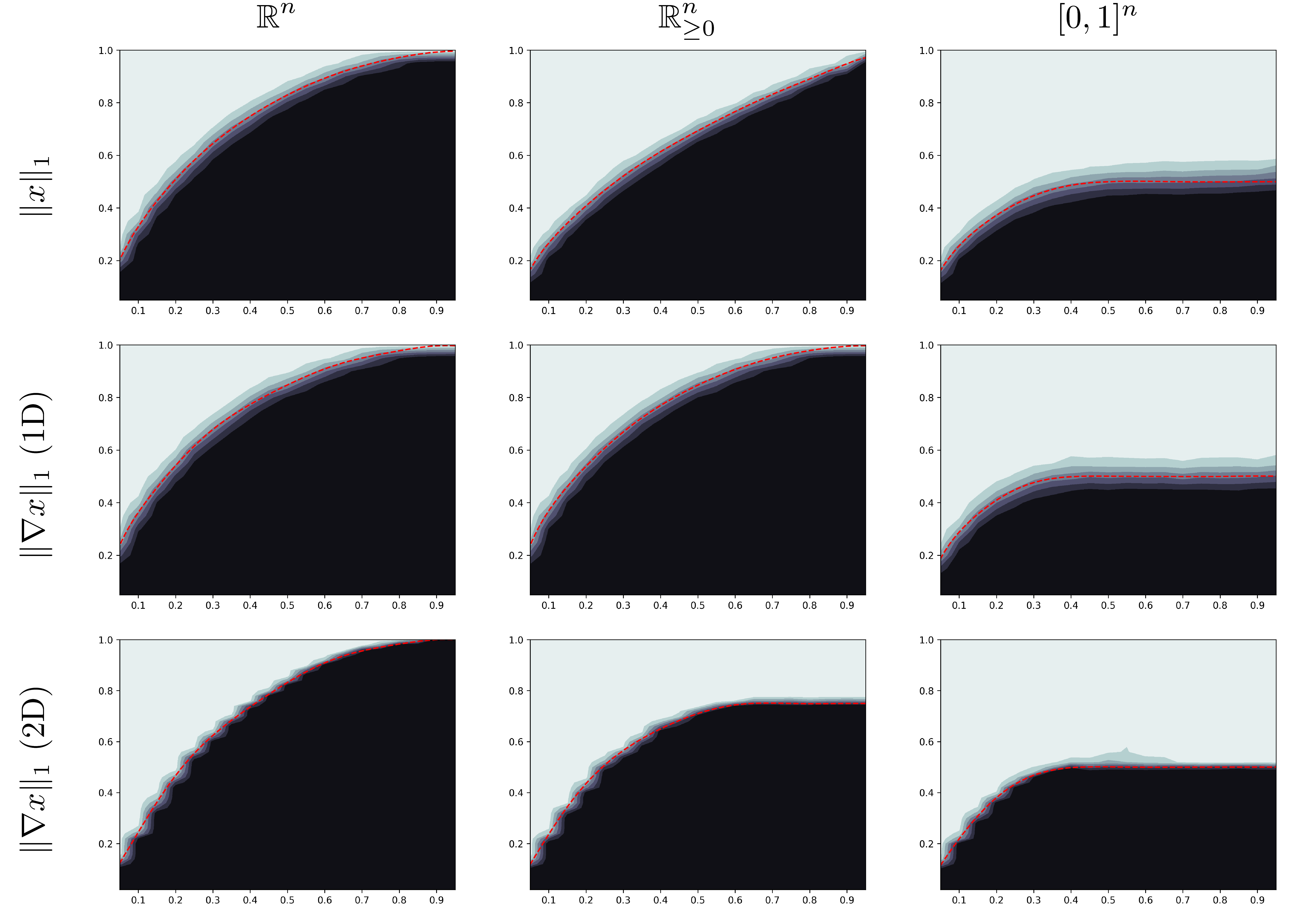}
  \caption{ \emph{Phase diagrams for random measurements.} 
  The gray value shows the empirical probability of uniqueness for each pair of parameters 
  (relative sparsity or relative gradient sparsity, \#measurements):
   black $\leftrightarrow 0\%$ uniqueness rate, white $\leftrightarrow 100\%$ uniqueness rate.
  Both regions are accurately separated by our approximation to the statistical dimension. 
  Rows from top to bottom show results for: $\ell_1$-, 1D $\TV$- and 2D $\TV$-minimization. 
  Columns for left to right show the signal/image entries: $\R^n$, $\R^n_+$ or $\{0,1\}^n$. 
  \label{fig:phase_transitions}}
\end{figure}
\begin{figure}
  \centering
  \includegraphics[width=\textwidth]{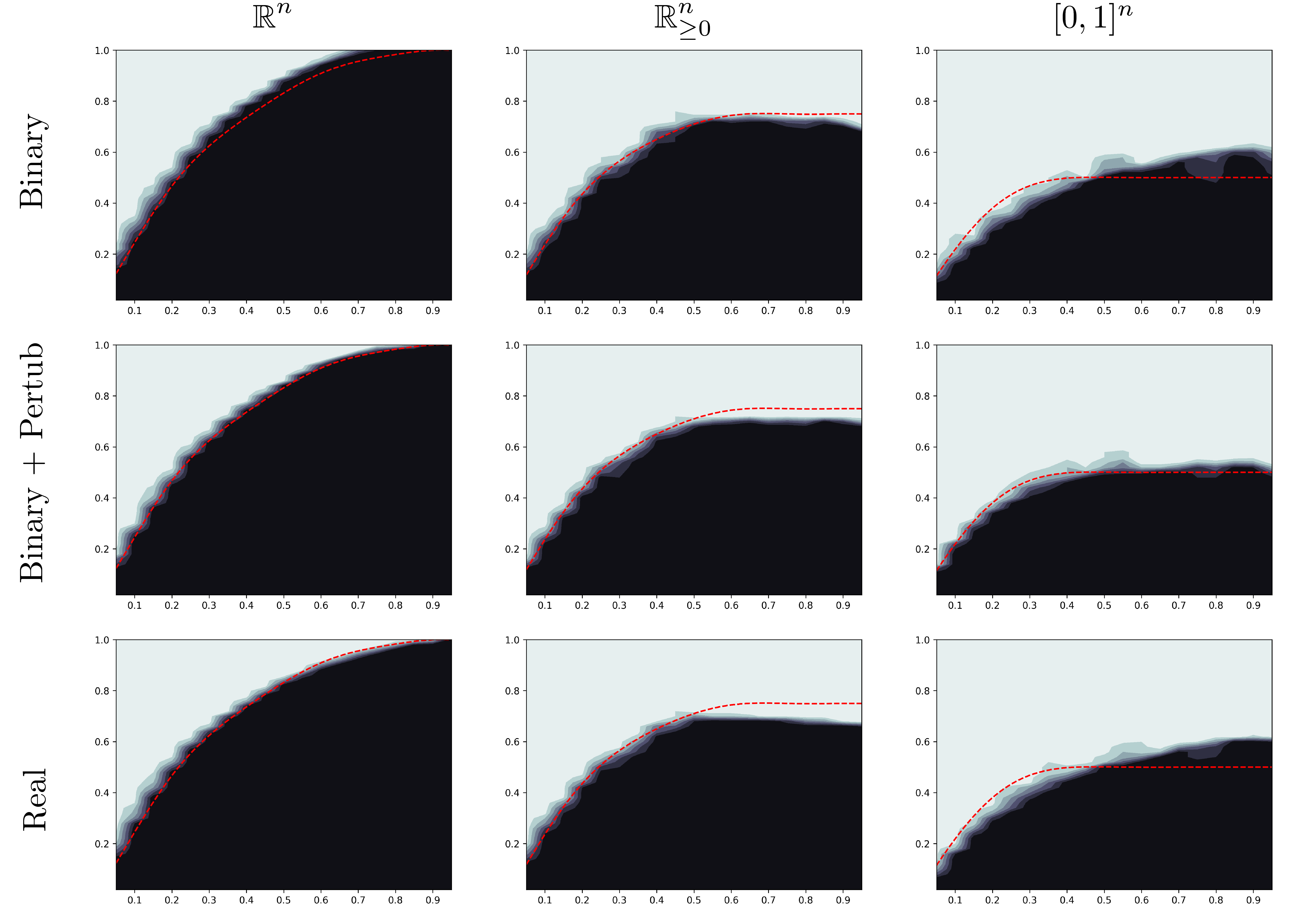}  
  \caption{ \emph{Phase diagrams for tomographic measurements.} As in  Figure \ref{fig:phase_transitions}
  we display the empirical probability of uniqueness in the case of \emph{tomographic measurements} for each pair of parameters 
  (relative sparsity or relative gradient sparsity, \#measurements):
   black $\leftrightarrow 0\%$ uniqueness rate, white $\leftrightarrow 100\%$ uniqueness rate.
   Rows from top to bottom show results for: binary, perturbed binary and standard matrices, i.e. with nonnegative real entries.
   The results demonstrate a remarkable agreement of the empirical phase transitions for tomographic recovery with the
   approximated curve based on the statistical dimension for random measurements.
   \label{fig:phase_transitions_tomo}}
\end{figure}
\subsubsection{Tomographic Measuremnts}
Making our experiments more general, we applied the idea from \cite{roux2014efficient} to reduce the influence of angles in tomographic reconstructions. The basic idea is, that each projection should carry the same amount of information independent of the angle. In Figure \ref{fig:tomo_rays} we illustrate the difference the effect on the
 tomographic projections when sensing an images embedded in a rectangle shape resp. to a circular shape. We emphasize that our testset for 2D contains only images embedded in a circular shape.
\begin{figure}
  \begin{subfigure}{0.45\textwidth}
    \centering
    \includegraphics[height=4cm]{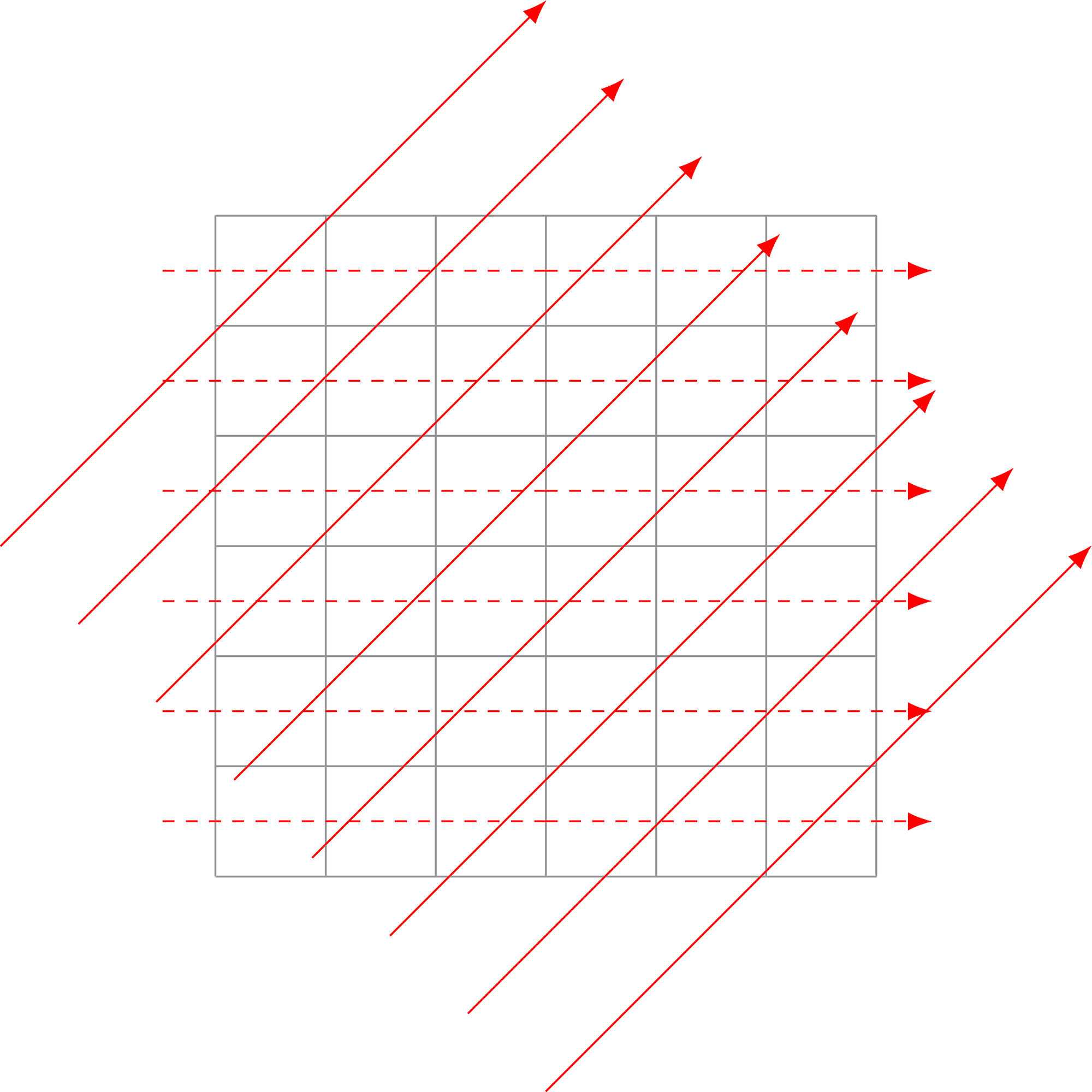}
    \subcaption{rectangle \label{fig:tomo_rectangle}}
  \end{subfigure}
  \begin{subfigure}{0.45\textwidth}
    \centering
    \includegraphics[height=4cm]{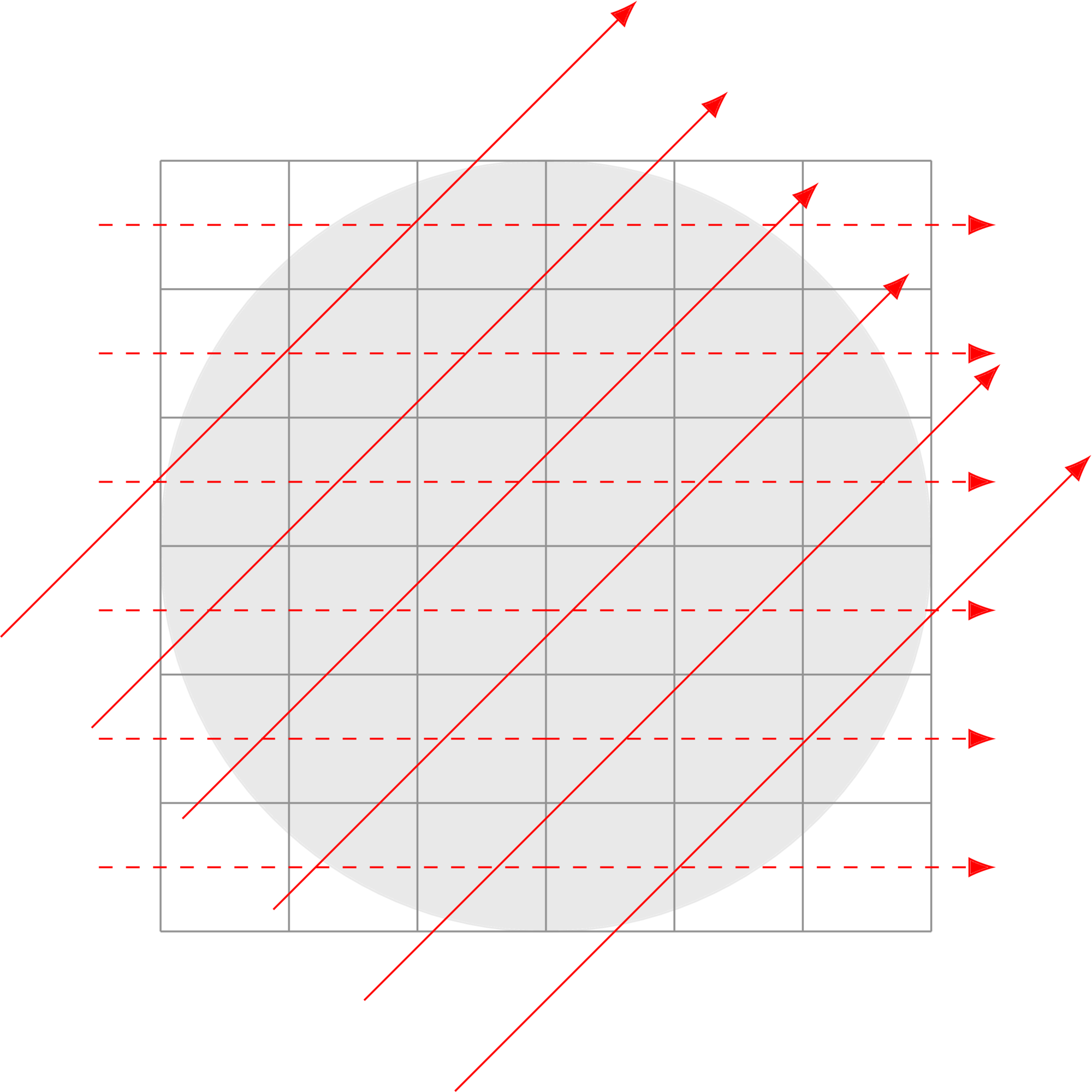}
    \subcaption{circle \label{fig:tomo_circle}}
  \end{subfigure}
  \caption{Illustration of parallel projections of a rectangle image (\subref{fig:tomo_rectangle}) and of an image embedded in a circular mask (\subref{fig:tomo_rectangle}).
  In  (\subref{fig:tomo_rectangle}) one needs at different angles a different amount of parallel rays to cover the whole image. For sensing an image in a circular mask
  the number of parallel rays covering the whole image would  be the same. \label{fig:tomo_rays}}
\end{figure}
\subsubsection{$\ell_1$-Minimization}
The phase diagrams for 2D images coincide with the ones for 1D signals, shown in Figure \ref{fig:phase_transitions}, first row,
and are omitted here. The two regions are accurately separated by the $\min_{\tau} J(\tau)$ curve.

\subsubsection{$\TV$-Minimization}
For fixed ambient dimension $n=64\cdot 64$ we choose a relative image sparsity $\frac{\|D\ol{x}\|_0 }{p} \in [0,1] $, with $D \in \R^{p \times n}$,
corresponding to a gradient sparse test image and also choose the number of measurements $m$. 
Measurements can be Gaussian or tomographic. We consider three types of tomographic matrices:
\begin{itemize}
\item binary, as described in \cite{roux2014efficient};
\item perturbed binary, that is we slightly perturb the nonzero entries above, in order to remove linear dependencies between columns, and
\item standard tomographic matrices, with nonnegative real entries.
We use the MATLAB routine \texttt{paralleltomo.m} from the AIR Tools package \cite{AirTools} that implements such a tomographic matrix for a arbitrary vector of angles. We choose equidistant angles,  set \texttt{N} $=64$ the image size and use the default value of \texttt{p}, i.e. the number of parallel rays for each angle \texttt{p = round(sqrt(2)*N)} to obtain a tomographic matrix of size $m\times n$.
\end{itemize}

Results are presented in Figure \ref{fig:phase_transitions}, last row, and Figure \ref{fig:phase_transitions_tomo}.
We show the empirical probability over the 10 repetitions.
The success rate of image reconstruction \emph{equals} the uniqueness result and is displayed by gray values: black $\leftrightarrow 0\%$ uniqueness rate, white $\leftrightarrow 100\%$ uniqueness rate.

All plots display a phase transition and thus exhibit regions where exact image reconstruction has probability equal or close to one.
All regions are accurately separated by our approximation to the statistical dimension.

%
%
%
\section{Conclusions}\label{sec:conclusions}
The present work was motivated by the need of accurately describing how much undersampling
is tolerable to uniquely recover a sparse or gradient sparse box constrained signal from tomographic measurements. 
Our results show sharp average-case phase transitions from non-uniqueness and non-recovery to uniqueness and exact recovery as guaranteed by compressed sensing for random measurements.
Moreover, we show that the phase transition occurs approximately at  the statistical dimension of the descent cone of the
structure enforcing objective, irrespectively of the employed measuring device. 
The approximation of the statistical dimension is computed empirically but cross-checked by the verifiable uniqueness test developed in this work. 

\section*{Appendix}
We collect few results from \cite{Rockafellar70convexanalysis,rockafellar1wets}.
The \textit{subdifferential} of  a function $f \colon \R^n \to \ol{\R}$ at a point $\ol{x}$ where $f(\ol{x}) \in \R$, is defined by
\begin{equation} \label{subdifferential_neu}
\partial f(\ol{x}) = \big\{p \in \R^{n} \colon f(x) - f(\ol{x}) \geq \langle p, x-\ol{x} \rangle,\quad \forall x \in \R^{n}\big\}.
\end{equation}
We set $\partial f(\ol{x}) = \emptyset$ if $f(\ol{x})$ is not finite.
A function $f \colon \R^{n} \to \ol{\R}$ is said to be
\textit{subdifferentiable} at $\ol{x}$ if
$\partial f(\ol{x}) \not = \emptyset$. Then the elements of $\partial f(\ol{x})$ are called
\textit{subgradients} of $f$ at $\ol{x}$.
For  a proper, convex function $f$ and
$\ol{x} \in \rint(\dom f)$, the subdifferential $\partial f(\ol{x})$ is nonempty. Furthermore,
$\partial f(\ol{x})$ is nonempty {\rm and} bounded if and only if
$\ol{x} \in \intr(\dom f)$. Hence only for $f_1$ and $f_4$ from \eqref{def-f1-l1} and \eqref{def-f6-TV-box}
$\partial f(\ol{x})$ is compact. For the remaining choices of $f$ in \eqref{def-f} $\partial f(\ol{x})$ is unbounded,
since in these cases our signals $\ol{x}$ lie on the boundary of  $\dom f = \R^n_+$ or  of 
$\dom f = \{0,1\}^n$.

The \emph{relative interior} of a non-empty convex set $C$ is the
interior relative to its affine hull,
\begin{equation*}
  \rint C = \big\{ x \in \aff C \colon \exists\veps > 0
  \text{ such that } (x + \veps\B(0))\cap \aff \,C \subset C \big\}.
\end{equation*}
The \emph{relative boundary} of $C$ is
\begin{equation*}
  \rbd C = \cl C \setminus \rint C.
\end{equation*}
We always have \cite[Cor. 6.6.2]{Rockafellar70convexanalysis}
\begin{equation} \label{eq:sum-relint}
\rint(C_1+C_2) = \rint(C_1) + \rint(C_2).
\end{equation}
Recall that the \emph{affine hull} of some set $C \subset \R^{n}$ is the
set of all affine combinations of its points,
$$
  \mrm{aff}\,C = \big\{ \lambda_{1} x^{1} + \dotsb + \lambda_{d} x^{d} 
  \colon \lambda_{1} + \dotsb + \lambda_{d} = 1,\;
  x^{1},\dotsc,x^{d} \in C
  \big\},
$$
while the \emph{conic hull} equals 
$$
\cone C = \big\{ \lambda_{1} x^{1} + \dotsb + \lambda_{d} x^{d} 
  \colon \lambda_{1}, \dots, \lambda_{d} \ge 0,\;
  x^{1},\dotsc,x^{d} \in C
  \big\}.
$$

\section*{Acknowledgment}
JK gratefully acknowledges the support by the German Science Foundation, grant GRK 1653.
\bibliographystyle{amsalpha}
\bibliography{tex/HDLiterature}
\vspace{2mm} \noindent \footnotesize
\begin{minipage}[b]{10cm}
Jan Kuske\\
Mathematical Imaging Group,\\
Heidelberg University, Germany.\\
Email: jan.kuske@iwr.uni-heidelberg.de
\end{minipage}

\vspace{2mm} \noindent \footnotesize
\begin{minipage}[b]{10cm}
Stefania Petra\\
Mathematical Imaging Group,\\
Heidelberg University, Germany.\\
Email: petra@math.uni-heidelberg.de
\end{minipage}
%
%
\end{document}